\documentclass[11pt,reqno]{amsart}
\usepackage{a4wide}
\usepackage{amssymb}
\usepackage{amsthm}
\usepackage{amsmath,todonotes}
\usepackage{amscd}
\usepackage{cancel}
\usepackage{verbatim}
\usepackage{color}
\usepackage[all]{xy}
\usepackage[mathscr]{eucal}
\usepackage{mathrsfs}
\usepackage{fullpage}
\usepackage{enumitem}
\usepackage{bbm}
\numberwithin{equation}{section}

\newtheorem{theorem}{Theorem}
\newtheorem{lemma}[theorem]{Lemma}
\newtheorem{corollary}[theorem]{Corollary}

\newcommand{\ppart}{\mathscr{P}}

\newcommand{\abs}[1]{\left|#1\right|}

\theoremstyle{remark}

\newtheorem*{remark}{Remark}
\newtheorem*{remarks}{Remarks}
\theoremstyle{definition}

\newtheorem*{definitionno}{Definition}
\numberwithin{theorem}{section} \numberwithin{equation}{section}

\newcommand{\pr}{\text {\rm pr}}

\newcommand{\R}{\mathbb{R}}
\newcommand{\C}{\mathbb{C}}

\newcommand{\Q}{\mathbb{Q}}

\newcommand{\Z}{\mathbb{Z}}
\newcommand{\N}{\mathbb{N}}
\newcommand{\SL}{{\text {\rm SL}}}

\newcommand{\sgn}{\operatorname{sgn}}
\renewcommand{\SS}{\mathbb{S}}

\newcommand{\re}{\textnormal{Re}}
\newcommand{\im}{\textnormal{Im}}

\def\H{\mathbb{H}}

\newcommand{\wt}{\kappa}

\newcommand{\HH}{\mathbb{H}}

\begin{document}
\title[Regularized inner products of meromorphic modular forms and higher Green's functions]{Regularized inner products of meromorphic modular forms and higher Green's functions}
\author{Kathrin Bringmann}
\address{\rm Mathematical Institute, University of Cologne, Weyertal 86-90, 50931 Cologne, Germany}
\email{kbringma@math.uni-koeln.de}
\author{Ben Kane}
\address{\rm Mathematics Department, University of Hong Kong, Pokfulam, Hong Kong}
\email{bkane@maths.hku.hk}
\author{Anna-Maria von Pippich}
\address{\rm Fachbereich Mathematik, Technische Universit\"at Darmstadt, Schlo{\upshape{\ss}}gartenstr. 7, 64289 Darmstadt, Germany }
\email{pippich@mathematik.tu-darmstadt.de}
\date{\today}


\medskip
\begin{abstract}
In this paper we study generalizations of Poincar\'e series arising from quadratic forms, which naturally occur as outputs of theta lifts. Integrating against them yields evaluations of higher Green's functions. For this we require a new regularized inner product, which is of independent interest.
\end{abstract}

\thanks{ The research of the first author is supported by the Alfried Krupp Prize for Young University Teachers of the Krupp foundation and the research leading to these results receives funding from the European Research Council under the European Union's Seventh Framework Programme (FP/2007-2013) / ERC Grant agreement n. 335220 - AQSER.  The research of the second author was supported by grants from the Research Grants Council of the Hong Kong SAR, China (project numbers HKU 27300314, 17302515, and 17316416).}
\maketitle

\section{Introduction and statement of results}\label{sec:intro}
While investigating the Doi-Naganuma lift, Zagier \cite{ZagierRQ} encountered interesting weight $2k$ cusp forms for $\SL_2(\Z)$ ($k\in \N_{\geq 2}$, $\ell\in\Z$) for $\ell = \delta >0$ defined by
\begin{equation}
\label{fkd}
f_{k,\ell}:=\sum_{\mathcal{A} \in \mathcal{Q}_{\ell}/\SL_2(\Z)} f_{k,\ell,\mathcal{A}}.
\end{equation}
Here $\mathcal{Q}_{\ell}$ is the set of integral binary quadratic forms of discriminant $\ell \in \Z$ and for  $\mathcal{A}$ an $\SL_2(\Z)$-equivalence class of quadratic forms of discriminant $\ell$, we set ($z\in\mathbb{H}$) 
\begin{equation*}
f_{\mathcal{A}}(z)=f_{k,\ell, \mathcal{A}}(z):= |\ell|^{\frac{k}{2}} \sum_{Q\in \mathcal{A}} Q(z,1)^{-k}.
\end{equation*}
Throughout we write $\delta>0$ for positive discriminants and let $-D<0$ denote negative discriminants.
Kohnen and Zagier \cite{KohnenZagierRational} showed, using a different normalization, that the even periods of $f_{k,\delta}$  are rational, and Kramer \cite{Kramer} proved that the $f_{k,\delta, \mathcal{A}}$ span the space of weight $2k$ cusp forms. Furthermore,  Kohnen and Zagier \cite{KohnenZagier} used the functions $f_{k,\delta}$ to construct a kernel function for the Shimura and Shintani lifts.  These may also be realized as theta lifts.

Roughly speaking, a theta lift is a map between modular objects in different spaces.  One begins with a theta kernel $\Theta(z,\tau)$,  which is modular in both variables. In our setting, both variables lie in $\H$ and $\Theta(z,\tau)$ has integral weight in $z$ and a half-integral weight in $\tau$.  Given a function $\tau\mapsto P(\tau)$ transforming with the same weight as $\Theta$ in the $\tau$-variable, one may then define the \begin{it}theta lift\end{it} of $P$ by taking the Petersson inner product $\langle \cdot,\cdot\rangle$ between $\Theta$ and $P$:
\[
\Phi(\Theta;P)(z):=\left<P,\Theta(z,\cdot)\right>.
\]

Niwa \cite{Niwa} wrote the Shimura and Shintani lifts as theta lifts by using a theta kernel corresponding to an integral quadratic form of signature $(2,1)$, which was later extended by Oda \cite{Oda} to signature $(2,n)$ for $n\in \mathbb{N}$. These lifts fit into the general framework of the theta correspondence between automorphic forms associated to two groups of a dual reductive pair \cite{Howe}.  Theta lifts have appeared in a variety of applications, including a relation of Katok and Sarnak \cite{KatokSarnak} between central values of $L$-functions and Fourier coefficients.  Paralleling the results in \cite{KatokSarnak},  the realization of $f_{k,\delta}$ as theta lifts gave the non-negativity of twisted central $L$-values \cite{KohnenZagier}.

Natural inputs for theta lifts  are Poincar\'e series.  These are defined, in the simplest case, for a translation-invariant function $\varphi$ (in the case of absolute convergence) as
\begin{equation*} 
\sum_{\gamma \in  \Gamma_{\infty}  \backslash \SL_2(\Z)} \varphi|_{\kappa}\gamma (\tau),
\end{equation*}
where $\Gamma_{\infty}:= \{\pm \left(\begin{smallmatrix}1&n\\0&1 \end{smallmatrix} \right)   : n \in \mathbb{Z}  \}$, $\kappa\in\frac{1}{2}\Z$ (throughout the paper we use $\kappa$ for arbitrary weight in $\frac12\Z$ and reserve $k$ for restricted weights), and $|_{\kappa}$ denotes the usual slash operator.  A natural choice for $\varphi$ is a term from the Fourier expansions of forms  in the space of automorphic forms in which one is interested. In this paper, we consider in particular half-integral weight modular forms and \begin{it}harmonic Maass forms\end{it}, which transform and grow like modular forms but instead of being meromorphic they are annihilated by the weight $\kappa$
{\it hyperbolic Laplace operator} (in the variable $z=x+iy\in\H$), defined by
\begin{equation}\label{Laplace}
\Delta_{\kappa} := -y^2\left(\frac{\partial^2}{\partial x^2}+\frac{\partial^2}{\partial y^2}\right)+i\kappa y\left(\frac{\partial}{\partial x}+i\frac{\partial}{\partial y}\right).
\end{equation}
We denote the Poincar\'e series constructed by choosing a particular function $\varphi$ from the Fourier expansions of these forms by $P_{k+\frac12,m}$ and $\mathcal{P}_{\frac32-k,m}$ (see \eqref{eqn:gDdef} and \eqref{eqn:PnD3/2-kdef} for the explicit definitions). This gives in particular four relevant cases: for positive weight one can average a cusp form coefficient or a coefficient that grows towards $i\infty$, while in negative weight one can define two kinds of Poincar\'e series, one that grows in the holomorphic part and one that grows in the non-holomorphic part (see \eqref{split} for the decomposition).

 We start with the case of positive weight and define (with $\tau=u+iv\in\H$) the theta kernel 
 \begin{equation}\label{eqn:thetadef1}
\Theta_k(z,\tau):=y^{-2k}v^{\frac{1}{2}}\sum_{D\in \Z}\sum_{Q\in\mathcal{Q}_D}Q(z,1)^k e^{-4\pi Q_{z}^2 v}e^{2\pi i D\tau}.
\end{equation}
Here, for $Q=[a,b,c]$, we set
\begin{equation}\label{eqn:Qzdef}
Q_{z}:=y^{-1}\left(a|z|^2+bx+c\right).
\end{equation}
 It is well-known that the function $z\mapsto \Theta_k(-\overline{z},\tau)$ is modular of weight $2k$ and $\tau\mapsto \Theta_{k}(z,\tau)$ is modular of weight $k+\frac12$. Hence taking the inner product in either variable yields a lift between integral and half-integral weights. For this, we define the following theta lift
\begin{equation*}
\Phi_k(f)(z):=\Phi(\Theta_k;f)(z).
\end{equation*}
Using as input positive weight cuspidal Poincar\'e series, one recovers the functions $f_{k,\delta}$:
\begin{equation*}
f_{k,\delta}=C_{k,\delta}\cdot\Phi_{k}\left(P_{k+\frac12,\delta}\right)
\end{equation*}
 with $C_{k,\delta}$ an explicit constant. By the Petersson coefficient formula, the holomorphic projection (recalled below) of the theta kernel $\Theta_k$ yields the generating function 
\[
\Omega_k(z,\tau):=\sum_{\delta> 0} \delta^{\frac{k-1}{2}}f_{k,\delta}(z)  e^{2\pi i \delta\tau}.
\]
Kohnen and Zagier \cite{KohnenZagier} proved that $z\mapsto\Omega_k(z,\tau)$ is a weight $2k$ cusp form and $\tau\mapsto\Omega_k(z,\tau)$ is a weight $k+\frac12$ cusp form.  By integrating in either variable, $\Omega_k$ yields theta lifts from weight $2k$ to $k+\frac12$ and from weight $k+\frac12$ to $2k$; these lifts turn out to yield an alternative construction of the well-known (first) Shintani \cite{Shintani} and Shimura \cite{Shimura} lifts.
Hereby, the idea underlying the holomorphic projection operator is simple. Suppose that $f$ is a weight $\kappa$ real-analytic modular form with moderate growth at cusps. Then $g\mapsto \langle g,f\rangle$ yields 
a linear functional on the space of weight $\kappa$ cusp forms.  Since the Petersson inner product is non-degenerate, this functional must be given by $\langle \cdot, F \rangle$ for some weight $\kappa$ cusp form $F$. This $F$ is essentially the holomorphic projection of $f$.

If one takes weakly holomorphic Poincar\'e series (i.e., Poincar\'e series which yield meromorphic modular forms with poles only at the cusps) as inputs of the theta lifts, one obtains, instead of the $f_{k,\delta}$'s, the analogous meromorphic modular forms $f_{k,-D}$ defined in \eqref{fkd}.  We note that some care is needed if the inputs are no longer cusp forms.  Although the naive definition of the inner product usually diverges when taking weakly holomorphic modular forms one can extend its definition,  and define a regularized theta lift that is meaningful for more general inputs; 
we describe this in Section \ref{sec:regularization}. To obtain the functions $f_{k,-D}$ as theta lifts, we use a regularization of Borcherds. The functions $f_{k,-D}$ were first constructed by  Bengoechea \cite{Bengoechea} in her thesis.
 \begin{theorem}\label{thm:liftfkd}
We have
\begin{equation*}
\Phi_k\left(P_{k+\frac{1}{2},-D}\right)=f_{k,-D}.
\end{equation*}
\end{theorem}
\begin{remarks}
\noindent

\noindent
\begin{enumerate}[leftmargin=*]
\item
The theta lift in Theorem \ref{thm:liftfkd} is a special case of a more general theta lift introduced by Borcherds in Theorem 14.3 of \cite{Bo1}.  We choose a Poincar\'e series as a distinguished input, while Borcherds had  more general inputs. Moreover, Borcherds unfolded against the theta function, while we apply the unfolding method to the Poincar\'e series.  As a result, the two approaches yield different representations of the functions $f_{k,-D}$.
\item
The theta lift $\Phi_k$ maps (parabolic) Poincar\'e series $P_{k+\frac12,\ell}$ to other types of Poincar\'e series; the functions $f_{k,\delta}$ are sums of the hyperbolic Poincar\'e series which appeared in previous work of Petersson \cite{Pe3} (see also \cite{ImOs}), while we see in \eqref{eqn:fkDPsi} that the $f_{k,-D}$ are sums of the elliptic Poincar\'e series defined by Petersson in \cite{Pe1}. This implies that they are elements of $\SS_{2k}$, the space of {\it  meromorphic cusp forms of weight $2k$} for $\SL_2(\Z)$, which are meromorphic modular forms that decay like cusp forms towards $i\infty$.
\end{enumerate}
\end{remarks}

We turn now to the case of negative weight. We use the theta kernel ($k\in \N_{\geq 2}$) 
 \begin{equation*}
\Theta^*_{1-k} (z,\tau):=v^k \sum_{D\in \Z}\sum_{Q\in \mathcal{Q}_D} Q_{z}Q(z,1)^{k-1} e^{-\frac{4\pi |Q(z,1)|^2 v}{y^2}}  e^{-2\pi i D\tau}.
\end{equation*}
\noindent
The function $z\mapsto \Theta_{1-k}^*(z,\tau)$ is modular of weight $2-2k$, and $\tau\mapsto \Theta^*_{1-k}(z,\tau)$ is modular of weight $\frac32-k$.  We set
\begin{equation*}
\Phi_{1-k}^*(f)(z):= \Phi(\Theta_{1-k}^*;f)\!\left(-\overline{z}\right).
\end{equation*}
 We then define negative-weight analogues of the functions $f_{k,\ell}$ (with $\ell\in\Z$), namely 
\begin{equation*}
\mathcal{F}_{1-k,\ell}:=\sum_{\mathcal{A} \in \mathcal{Q}_{\ell}/\SL_2(\Z)}  \mathcal{F}_{1-k,\ell,\mathcal{A}},
\end{equation*}
where
\begin{equation}\label{eqn:Gdef}
\mathcal{F}_\mathcal{A}(z)=\mathcal{F}_{1-k,\ell,\mathcal{A}}(z):=\sum_{Q\in \mathcal{A} }\mathbb{P}_{1-k,\ell,Q}(z)
\end{equation}
with
\begin{equation}\label{defineP}
\mathbb{P}_{1-k,\ell,Q}(z):=\frac{i(-1)^k}{2}|\ell|^{\frac{1-k}{2}}\sgn\left(Q_{z}\right) Q\left(z,1\right)^{k-1} \beta\left(\frac{ \ell y^2}{\left|Q\left(z,1\right)\right|^2_{\phantom{-}}};k-\frac{1}{2},\frac{1}{2}\right).
\end{equation}
Here $\beta\left(Z;a,b\right)$ denotes the {\it incomplete $\beta$-function}, which is defined for $a,b\in \C$ satisfying $\re(a)$, $\re(b)>0$ by $\beta\left(Z;a,b\right):=\int_{0}^Z t^{a-1}\left(1-t\right)^{b-1}dt$. Note that we can also write the incomplete $\beta$-function in terms of the hypergeometric function $_2F_1$ (see \eqref{equ-beta-hypergeom}). 

We recall some of the properties of these functions for $\ell=\delta>0$.  The $\mathcal{F}_{1-k,\delta,\mathcal{A}}$, with a different normalization, were investigated by Kohnen and the first two authors in \cite{BKW}, and a variant of these functions was studied by H\"ovel \cite{Hoevel} for $k=1$.
It turns out that they are locally harmonic Maass forms.  Locally harmonic Maass forms allow jump  singularities in the upper half-plane. These functions and their higher-dimensional analogues also appeared as theta lifts in both physics and mathematics -- see the work of Angelantonj, Florakis, and Pioline \cite{PiolineOneLoop} for the former and the work of Viazovska and the first two authors \cite{BKM} for the latter. Namely, in analogy to the positive weight case, we have \cite{BKM, Hoevel}
$$
\mathcal{F}_{1-k,\delta}=C_{1-k,\delta}\cdot \Phi_{1-k}\left(\mathcal{P}_{\frac32 -k,\delta}\right),
$$
with $C_{1-k,\delta}$ an explicit constant and $\mathcal{P}_{3/2 -k,\delta}$ defined in \eqref{eqn:PnD3/2-kdef}.  In addition to their relationship via theta lifts, the functions $\mathcal{F}_{\mathcal A}$ are connected to the functions $f_{\mathcal A}$ via the differential operators $\xi_{2-2k}$ and $\mathcal{D}^{2k-1}$, where
\begin{equation}\label{XiD}
\xi_{\kappa}:=2iy^{\kappa}\overline{\frac{\partial}{\partial \overline{z}}}\quad\textnormal{ and }\quad\mathcal{D}:=\frac{1}{2\pi i } \frac{\partial}{\partial z}.
\end{equation}
Specifically, we have
\begin{align*}
\xi_{2-2k}\left(\mathcal{F}_{\mathcal{A}}\right) =
\mathcal{C}_{1,k,\delta}\cdot f_{\mathcal{A}},
\quad\quad
\mathcal{D}^{2k-1}\left(\mathcal{F}_{\mathcal{A}}\right)=
\mathcal{C}_{2,k,\delta}\cdot f_{\mathcal{A}},
\end{align*}
where the $\mathcal{C}_{j,k,\delta}$ are explicit constants.

It is unusual for a harmonic Maass form to map to a constant multiple of the same function under $\xi_{2-2k}$ and $\mathcal{D}^{2k-1}$.
However, given their uniform definition in \eqref{eqn:Gdef}, it is not a surprise to find out that for discriminants $-D<0$, the functions $\mathcal{F}_{1-k,-D,\mathcal{A}}$ have many properties similar to those of $\mathcal{F}_{1-k,\delta,\mathcal{A}}$.  As a difference between negative and positive discriminants, note that for negative discriminants the functions have poles instead of jump singularities.  We call functions that behave like harmonic Maass forms away from singularities of this type \begin{it}polar harmonic Maass forms\end{it}.
\begin{theorem}\label{thm:Gpolar}
\noindent

\noindent
\begin{enumerate}[leftmargin=*,label={\rm(\arabic*)}]
		\item We have
		\[
		\Phi_{1-k}^* \left( \mathcal{P}_{\frac{3}{2}-k,-D}\right) =\mathcal{F}_{1-k,-D}.
		\]
		\item For $\mathcal{A}\in\mathcal{Q}_{-D}\slash\SL_2(\Z)$, the functions $\mathcal{F}_{\mathcal{A}}$ are weight $2-2k$ polar harmonic Maass forms whose only singularities occur at $\tau_Q$ for $Q\in \mathcal{A}$; here $\tau_Q\in\mathbb H$ is the unique solution to $Q(z,1)=0$.  Furthermore, we have
		\begin{align}\label{eqn:xiDG}
		\xi_{2-2k}\left(\mathcal{F}_{\mathcal{A}}\right)& = f_{\mathcal{A}},& \mathcal{D}^{2k-1}\left(\mathcal{F}_{\mathcal{A}}\right)&= -\frac{(2k-2)!}{(4\pi)^{2k-1}} f_{\mathcal{A}}.
		\end{align}
	\end{enumerate}
\end{theorem}
\begin{remarks}
\noindent

\noindent
\begin{enumerate}[leftmargin=*]
\item
The difference in the singularities of $\mathcal{F}_{1-k,\ell,\mathcal{A}}$ for discriminants $\ell>0$ and $\ell<0$ comes from the sign factor in \eqref{defineP}.
For $\ell>0$, $Q_z=0$ along a geodesic $S_Q$, and the function ``jumps'' as one crosses from one side of $S_Q$ to the other.  For $\ell<0$, $\sgn(Q_z)\neq 0$ and $\sgn(Q_z)$ is independent of $z$; namely, $\sgn(Q_z)=1$ for all $z\in\H$ if $Q$ is positive-definite and $\sgn(Q_z)=-1$ for all $z\in\H$ if $Q$ is negative-definite.
\item
A key step in proving Theorem \ref{thm:Gpolar} (2) is to relate $\mathcal F_\mathcal{A}$ to the higher Green's functions $G_k$ defined in Subsection \ref{sec:Greens} (see Corollary \ref{diffop}).  These have appeared in a number of interesting applications, and their evaluations at pairs of CM-points has been intensively studied. Values of higher Green's functions at CM-points are conjectured to be roughly logarithms of algebraic numbers, and a number of cases are known by work of Mellit \cite{Mellit} and Viazovska \cite{Vi}.
\end{enumerate}
\end{remarks}

Let us now return to the positive weight cusp forms $f_{k,\delta,\mathcal{A}}$. Integrating against them gives cycle integrals. To be more precise, we have
\[
\left<f,f_{k,\delta,\mathcal{A}}\right>=\mathcal{C}_{k,\delta}\int_{\Gamma_{Q_0}\,\backslash \, S_{Q_0}}f(z)Q_0(z,1)^{k-1}dz,
\]
where  $\mathcal C_{k,\delta}\in\R$ is an explicit constant, $Q_0 \in \mathcal{A}$ is arbitrary, $S_{Q_0}$ is an oriented geodesic joining the two real roots of $Q_0$, and $\Gamma_{Q_0}\subset \SL_2(\R)$ is the stabilizer group of $Q_0$. These cycle integrals then occur as coefficients of the (first) Shintani lift. In this paper we take the Petersson inner product of $f_{\mathcal A}$ with other meromorphic cusp forms.

Since the classical inner product diverges, one needs to regularize it.  In addition to their inherent interest, extensions of Petersson's inner product yield applications to other areas, including generalized Kac--Moody algebras \cite{GritsenkoNikulin} and the arithmetic of Shimura varieties \cite{BrY}.  Those applications used a regularization of Petersson \cite{Pe2}, later rediscovered and generalized by Borcherds \cite{Bo1} and Harvey--Moore \cite{HM}.  However, Petersson's inner product $\langle f,f \rangle$ for any (non-cuspidal) $f\in \SS_{2k}$ always diverges (see Satz 1 of \cite{Pe2}), so one cannot use it to extend the classical inner product to an inner product on any larger subspace.  For the application in this paper, we introduce a new regularized inner product, again denoted by $\left<\cdot,\cdot\right>$ and formally defined in \eqref{eqn:OurReg} below, which extends the domain of the inner product to include all meromorphic cusp forms.

\begin{theorem}\label{thm:innerconverge}
	The regularized inner product $\left<f,g\right>$ exists for all $f,g\in\SS_{2k}$. It is Hermitian, and it equals Petersson's regularized inner product whenever his exists.
\end{theorem}

\indent
We next consider an application of the inner product to higher Green's functions.  To state the formula we let $\omega_{\varrho}$ be the size of the stabilizer $\Gamma_{\varrho}$ of $\varrho\in\H$ with respect to the action of $\operatorname{PSL}_{2}(\Z)$.  
We require the \begin{it}elliptic expansion\end{it} of a meromorphic modular form $f$ around $\varrho\in\H$, namely
\begin{equation}\label{eqn:fEllExp}
 f(z)=\left(z-\overline{\varrho}\right)^{-2k}\sum_{n\gg -\infty} c_{f,\varrho}(n) X_{\varrho}(z)^n, \qquad\text{with}
 \qquad X_{\varrho}(z):=\frac{z-\varrho}{z-\overline{\varrho}}.
\end{equation}
Furthermore, set 
\begin{equation}\label{eqn:bkndef}
b_{k,n}:= \frac{(-1)^k (2k-2)!}{2^{3k-2}(k-1)!}
\begin{cases} \frac{1}{n!}& \text{if }n\geq k-1,\\
\frac{1}{(2k-2-n)!} &\text{if }n<k-1.
\end{cases}
\end{equation}
Throughout, we let
\begin{equation}\label{raising}
R_\kappa:=2i\frac{\partial }{\partial z} +\frac{\kappa}{y}
\end{equation}
be the {\it Maass raising operator}, and denote repeated raising by $R_{\kappa}^{n}:=R_{\kappa+2n-2}\circ\cdots\circ R_{\kappa}$.
\begin{theorem}\label{generalint}
If $Q_0\in \mathcal{A}\in\mathcal{Q}_{-D}/\SL_2(\Z)$ and $f\in \mathbb{S}_{2k}$ with poles in the $\SL_2(\Z)$-orbits of $\mathfrak{z}_1,\dots,\mathfrak{z}_r$ with $\mathfrak{z}_{\ell}=\mathbbm x_\ell+i\mathbbm y_\ell \neq \tau_Q$ for all $Q\in\mathcal{A}$ and $\ell\in \{1,\dots,r\}$, then 
\begin{multline*}
\left<f,f_{\mathcal{A}}\right>=\frac{\pi}{\omega_{\tau_{Q_0}}} \sum_{\ell=1}^r \frac1{\omega_{\mathfrak{z}_\ell}}
\Bigg(\sum_{n\geq k} b_{k,n-1}\mathbbm{y}_\ell^{-2k+n} c_{f,\mathfrak{z}_{\ell}}(-n)  R_{0}^{n-k}\left(G_k(z,\tau_{Q_0})\right)\\
+ \sum_{n=1}^{k-1} b_{k,n-1} \mathbbm{y}_\ell^{-n}  c_{f,\mathfrak{z}_{\ell}}(-n)  \overline{R_{0}^{k-n}\left(G_k(z,\tau_{Q_0})\right)}\Bigg).
\end{multline*}
\end{theorem}
Particularly interesting is the following special case.
\begin{corollary} \label{cor:Greensinner}
For every $Q_j\in \mathcal{A}_j\in \mathcal{Q}_{-D_j}\slash\SL_2(\Z)\  (j=1,2)$ with $\mathcal{A}_{1}\neq \mathcal{A}_2$ we have
	$$
	\left<f_{\mathcal{A}_1},f_{\mathcal{A}_2}\right>=
\pi\,b_{k,k-1} \frac{G_k\!\left(\tau_{Q_1},\tau_{Q_2}\right)}{\omega_{\tau_{Q_1}}\omega_{\tau_{Q_2}}}.
	$$
\end{corollary}

\begin{remarks}
	\noindent
\begin{enumerate}[leftmargin=*]
\item 
For arbitrary $z_1,z_2\in\H$, which are not necessarily CM-points, one may also realize $G_k(z_1,z_2)$ as an inner product. In order to obtain such a relation, one replaces $f_{\mathcal{A}_j}$ with the more general functions $\Psi_{2k,-k}(\cdot,z_j)$, defined in \eqref{eqn:Psidef} below, which have poles at $z_1,z_2\in\H$. Furthermore,  since Theorems \ref{thm:innerconverge} and \ref{generalint}  can be generalized to arbitrary congruence subgroups, similar relations can be established for the corresponding Green's functions associated to these subgroups. 
\item
Given the interest in $G_{k}$ evaluated at CM-points, one may wonder what further implications Corollary \ref{cor:Greensinner} may have. Possible future directions of study along these lines are discussed in Section \ref{sec:future} below.  The relation between higher Green's functions and inner products in Corollary \ref{cor:Greensinner} leads one to search for connections with geometry. In the case $k=1$, which is excluded here, Gross and Zagier related the Green's function evaluated at CM-points to the height pairing of certain Heegner points on modular curves (see Proposition 2.22 in Section II of \cite{GZ}).
 This has been generalized to higher $k$ by Zhang, who defined a global height pairing between
CM-cycles in certain Kuga--Sato varieties using arithmetic intersection theory,
as developed by Gillet and Soul\'e \cite{GilletSoule}.
The archimedean part of this height pairing is then given by the values of higher Green's functions evaluated at CM-points (see Propositions 3.4.1 and 4.1.2 of \cite{Zhang}). 
\item
Although we restrict ourselves in the introduction to the case $\mathfrak{z}_{\ell}\neq \tau_Q$ in Theorem \ref{generalint}, and correspondingly $\mathcal{A}_1\neq \mathcal{A}_2$ in Corollary \ref{cor:Greensinner}, this is only done for convenience of notation. By replacing the Green's function with a regularized version, we obtain a more general version of Theorem \ref{generalint} in Theorem \ref{thm:innerGreensGeneral} below, and consequently an extension of Corollary \ref{cor:Greensinner}.
\end{enumerate}
\end{remarks}

The paper is organized as follows. In Section \ref{sec:prelim} we recall basic geometric facts and certain special functions, and introduce the relevant modular objects. In Section \ref{sec:regularization} we study regularized inner products and prove Theorem \ref{thm:innerconverge}. In Section \ref{sec:theta} we investigate theta lifts, proving Theorem \ref{thm:liftfkd} and Theorem \ref{thm:Gpolar} (1). Theorem \ref{thm:Gpolar} (2) is established while studying the functions $\mathcal{F}_\mathcal{A}$ in  Section \ref{sec:FA}. In Section \ref{sec:residue} we compute regularized inner products in order to prove Theorem \ref{generalint} and Corollary \ref{cor:Greensinner}.  We conclude the paper with a discussion of natural questions in Section \ref{sec:future}.

\section*{Acknowledgements}

\noindent The authors thank Stephan Ehlen, Jens Funke, Pavel Guerzhoy, Steffen L\"obrich, and Tonghai Yang for helpful conversations, and  Shaul Zemel for discussions and finding an error in a previous version. Moreover they thank the anonymous referees for comments improving the exposition of the paper.

\section{Preliminaries}\label{sec:prelim}

\subsection{CM-points and the hyperbolic distance}

For a positive-definite $Q=[a,b,c]\in\mathcal{Q}_{-D}$ (with $a>0$), we denote the associated CM-point by
\begin{equation}\label{eqn:yQval}
	\tau_Q=u_Q+iv_Q, \qquad\text{with}\qquad u_Q=-\frac{b}{2a} \quad{\text{ and }}\quad v_Q=\frac{\sqrt{D}}{2a}.
\end{equation}
We note that for $z=x+iy\in\H$, with $X_{\tau_Q}$ defined in \eqref{eqn:fEllExp}, we have
\begin{equation}\label{rewriteQ}
	 Q(z,1) =\frac{\sqrt{D}}{2v_Q} \left(z-\overline{\tau_Q}\right)^2 X_{\tau_Q}(z).
\end{equation}
Moreover, for $Q\in \mathcal{Q}_{\ell}$, we often make use of the identity
\begin{equation}\label{eqn:Qrewrite}
	y^{-2}\left|Q(z,1)\right|^2 =Q_{z}^2+\ell,
\end{equation}
with $Q_z$ given in \eqref{eqn:Qzdef}. The quantity $Q_z$ naturally occurs when computing the hyperbolic distance $d(z,\mathfrak{z})$ between $z$ and $\mathfrak{z}=\mathbbm{x}+i\mathbbm{y}\in \H$, which is expressed through
\begin{equation}
\label{eqn:coshgen}
	\cosh\left(d(z,\mathfrak{z})\right) = 1+\frac{\left|z-\mathfrak{z}\right|^2}{2y\mathbbm{y}}
\end{equation}
(see p.~131 of \cite{Beardon}). In particular, when $\mathfrak{z}$ is a CM-point $\tau_Q$ with $Q\in\mathcal{Q}_{-D}$, we have the equality
\begin{equation}\label{eqn:coshQz/D}
	\cosh\left(d\left(z,\tau_{Q}\right)\right) = \frac{Q_z}{\sqrt{D}}.
\end{equation}
The combination of \eqref{eqn:Qrewrite} and \eqref{eqn:coshQz/D} gives
\begin{equation}\label{eqn:coshrat}
\left(1-\cosh(d(z,\tau_Q))^{2}\right)^{-1}=-\frac{Dy^2}{|Q(z,1)|^2}.
\end{equation}
Finally, for $z\in \H$ (and fixed $\varrho\in\H$) we set
\begin{equation*}
	r_{\varrho}(z):=\tanh\left(\frac{d(z,\varrho)}{2} \right)= \left|X_{\varrho}(z)\right|.
\end{equation*}
Here the last equality follows from the half-argument formula 
\[
\tanh\left(\frac{Z}{2}\right) = \sqrt{\frac{\cosh(Z)-1}{\cosh(Z)+1}}
\]
combined with \eqref{eqn:coshgen} and $\left|z-\mathfrak{z}\right|^2+4y\mathbbm{y}=\left|z-\overline{\mathfrak{z}}\right|^2$. Using this half-argument formula once again, equation \eqref{eqn:coshQz/D} implies that 
\begin{equation}\label{1r}
1-r_{\tau_Q}(z)^2= \frac{2}{\cosh\!\left(d\!\left(z,\tau_Q\right)\right) +1}=\frac{2\sqrt{D}}{Q_z+\sqrt{D}}.
\end{equation}
\subsection{Properties of hypergeometric functions}
In this subsection we recall relations between the hypergeometric function and other functions, as well as its transformations that are required for this paper.

For $Z\in\mathbb{C}$ with $|Z|<1$ the hypergeometric function is defined by the series
\begin{align}\label{2F1}
{}_2 F_1\left(a,b;c;Z\right):=\sum_{n\geq 0} \frac{(a)_n(b)_n}{n!(c)_n} Z^n,
\end{align}
with parameters $a,b,c\in\mathbb{C}$, $c$ not a non-negative integer, and $(a)_n:=\prod_{j=0}^{n-1}(a+j)$. Outside the disk $|Z|<1$, the hypergeometric function is defined by analytic continuation.
Using the symmetry in \eqref{2F1}, one directly sees that
\begin{equation}
{}_2 F_1\left(a,b;c;Z\right)={}_2 F_1\left(b,a;c;Z\right).
\label{2F1sym}
\end{equation}
Furthermore, by 15.4.6 of \cite{NIST} we have
\begin{equation}
{}_2 F_1\left(a,b;b;Z\right)=\left(1-Z\right)^{-a}.
\label{2F1 special}
\end{equation}
We also require the following transformation law from 15.8.1 of \cite{NIST}, valid when $1-Z\notin\R^-$:
\begin{equation}\label{2F1tr}
{}_2 F_1\left(a,b;c;Z\right)=\left(1-Z\right)^{-b}{}_2 F_1\left(c-a,b;c;\frac{Z}{Z-1}\right).
\end{equation}
By 15.5.3 of \cite{NIST} (with $n=1$) we have
 		\begin{equation}\label{diff2F1}
 		\frac{\partial}{\partial Z}\left(Z^{a}\,{_2F_1}\left(a,b;c;Z\right)\right)=aZ^{a-1}\,{_2F_1}\left(a+1,b;c;Z\right).
 		\end{equation}
By 
8.17.7 of \cite{NIST} the
 hypergeometric function is related to the incomplete $\beta$-function via
\begin{equation}\label{equ-beta-hypergeom}
\beta(Z;a,b)=\frac{Z^a}{a}\, {}_2F_1\left(a,1-b;a+1;Z\right).
\end{equation}
Using 15.8.14 of \cite{NIST} then implies that for $1-Z\not\in\R^-$ we have 
\begin{equation}\label{betatran}
	\beta(Z;2k-1,1-k)=2^{2k-2}i(-1)^k\beta\left(\frac{Z^2}{4(Z-1)}; k-\frac12, \frac{1}{2}\right).
\end{equation}

\subsection{Polar harmonic Maass forms and their elliptic expansions}

For $\gamma=\left(\begin{smallmatrix}a &b\\c &d\end{smallmatrix}\right)\in\SL_2(\Z)$ and $f\colon \H\to\C$, the weight $\wt\in\frac{1}{2}\Z$ {\it slash-action} is defined by
$$
f|_{\wt}\gamma(\tau):= (c\tau +d)^{-\kappa} f(\gamma \tau) \begin{cases}
                        1 & \text{ if } \kappa \in \Z,\\
\left(\frac{c}{d}\right) \varepsilon_d^{2\kappa} & \text{ if } \kappa \in \frac{1}{2}\Z \backslash \Z\text{ and }\gamma\in\Gamma_0(4),                       \end{cases}
$$
with the extended Legendre symbol ($\frac{\cdot}{\cdot}$) and 
\[
\varepsilon_d:=\begin{cases} 1 &\text{if }d\equiv 1\pmod{4},\\ i&\text{if }d\equiv 3\pmod{4}.\end{cases}
\]
We assume throughout that $N\in\N$ is divisible by $4$ if $\kappa\in\frac12+\Z$.

\begin{definitionno}
For $N\in\N$, a \begin{it}weight $\kappa\in \frac{1}{2}\Z$ polar harmonic Maass form on $\Gamma_0(N)$\end{it} is a real-analytic function $\mathcal{M} \colon  \H \rightarrow \C$ that satisfies the following conditions, outside finitely many singularities in $\Gamma_0(N)\backslash(\H\cup \Q \cup \{i\infty\})$:
\noindent

\noindent
\begin{enumerate}[leftmargin=*]
\item
For all $\gamma\in\Gamma_0(N)$ we have $\mathcal{M}|_{\kappa} \gamma =\mathcal{M}$.
\item
We have $\Delta_\kappa(\mathcal{M})=0$, with the hyperbolic Laplacian defined in \eqref{Laplace}.
\item
For every $\varrho\in\H$ there exists $n\in\N$ such that $(\tau-\varrho)^{n}\mathcal{M}(\tau)$ is bounded for $r_{\varrho}(\tau)\ll_{\mathcal{M}} 1$. We say that $\mathcal{M}$ has a \begin{it}singularity of finite order\end{it} at $\varrho$ if this condition is satisfied.
\item
The function $\mathcal{M}$ grows at most linear exponentially at the cusps.
\end{enumerate}
If the only singularities of $\mathcal{M}$ lie at the cusps, then $\mathcal{M}$ is a \begin{it}harmonic Maass form\end{it}.
\end{definitionno}
For $\kappa<1$, the Fourier expansion of a polar harmonic Maass form at $i\infty$ has a natural splitting of the shape
\begin{equation}
\label{split}
\mathcal{M}(\tau)=\mathcal{M}^+(\tau)+\mathcal{M}^-(\tau),
\end{equation}
where the {\it holomorphic} and {\it non-holomorphic parts} (at $i\infty$) are defined by the following series, that converge for $v$ sufficiently large, with $c_{\mathcal M}^{\pm}(n)\in\C$:
\begin{align*}
\mathcal{M}^+(\tau)&:=\sum_{n\gg -\infty} c^+_{\mathcal{M}}(n) e^{2\pi i n \tau},\\
\mathcal{M}^-(\tau)&:=c^-_{\mathcal{M}}(0) v^{1-\kappa}+\sum_{0\neq n\ll \infty} c^-_{\mathcal{M}}(n) \Gamma\left(1-\kappa,-4\pi n v\right)e^{2\pi i n\tau}.
\end{align*}
Here $\Gamma(r,Z):=\int_{Z}^{\infty} t^{r-1} e^{-t} dt$ denotes the \begin{it}incomplete gamma function\end{it}. 
Expansions of this type also exist for $\kappa\geq 1$, but the term in $\mathcal{M}^-$ containing $v^{1-\kappa}$ is replaced with a logarithmic term for $\kappa=1$, and more care is needed for the terms containing an incomplete gamma function when both parameters are negative. There are also similar expansions at the other cusps. One reason that the splitting into holomorphic and non-holomorphic parts is natural is that the holomorphic part is annihilated by the operator $\xi_{\kappa}$ defined in \eqref{XiD}.  Both parts can have singularities; the singularities in the holomorphic part are poles, while one can determine the kind of singularities in the non-holomorphic part by noting that its image under $\xi_{\kappa}$ is meromorphic.  The terms in the expansion which grow as $v\to\infty$ are called the \begin{it}principal part of $\mathcal{M}$ \end{it}(at $i\infty$); namely, for $\kappa<1$ these are the terms in $\mathcal{M}^+$ with $n<0$ and those terms in $\mathcal{M}^-$ with $n\geq 0$.   The coefficients $c_{\mathcal M}^-$ are closely related to coefficients of meromorphic modular forms of weight $2-\kappa$, following from the fact that if $\mathcal{M}$ is modular of weight $\kappa$, then $\xi_{\kappa}(\mathcal{M})$ is modular of weight $2-\kappa$. Thus $\xi_{\kappa}$ maps weight $\kappa$ polar harmonic Maass forms to weight $2-\kappa$ meromorphic modular forms.

Solving the second-order differential equation coming from \eqref{Laplace}, one obtains an elliptic expansion of polar harmonic Maass forms that parallels the expansion \eqref{eqn:fEllExp} for meromorphic cusp forms. The resulting expansion is given in Proposition 2.2 of \cite{BKweight0}, which appears, as Pioline later pointed out, as a special case of Theorem 1.1 of \cite{Fay}. To describe it, under the restriction $0\leq Z<1$, $a\in\N$, and $b\in\Z$, we set 
\begin{equation}\label{eqn:beta0def}
\beta_0\left(Z; a,b\right):=\beta\left(Z; a,b\right)-\mathcal{C}_{a,b} \hspace{7mm} \text{ with }\hspace{7mm}
\mathcal{C}_{a,b}:=\sum_{\substack{0\leq j\leq a-1\\ j\neq -b}} \binom{a-1}{j}\frac{(-1)^j}{j+b}.
\end{equation}
Making the change of variables $t \mapsto 1-t$ in the integral representation and then applying the Binomial Theorem, we obtain
\begin{equation}\label{betid}
	\beta_0 (Z;a,b) = \sum_{ \substack{0\leq j \leq a-1 \\ j\neq -b}} \binom{a-1}{j} \frac{(-1)^{j+1}}{j+b} (1-Z)^{j+b} +\delta_{1-a\leq b \leq 0} \binom{a-1}{-b} (-1)^{b+1} \log (1-Z).
\end{equation}
Here for a property $S$, $\delta_S=1$ if $S$ is satisfied and $\delta_S=0$ otherwise.

 For every $\varrho\in \H$, a polar harmonic Maass form $\mathcal{M}$ of weight $2-2k$ (or more generally any function $\mathcal{M}$ that is annihilated by $\Delta_{2-2k}$ and has a singularlity of finite order at $\varrho$), there exist $c_{\mathcal{M},\varrho}^{\pm}(n)\in\C$ such that for $r_{\varrho}(z)\ll_{\varrho} 1$ one has
\begin{equation}\label{eqn:expw}
\mathcal{M}(z)=\left(z-\overline{\varrho}\right)^{2k-2}\left(\sum_{n\gg -\infty}c_{\mathcal{M},\varrho}^+(n) X_{\varrho}(z)^n + \sum_{n\ll\infty }c_{\mathcal{M},\varrho}^-(n) \beta_0\left(1-r_{\varrho}(z)^2;2k-1,-n\right) X_{\varrho}(z)^n\right).
\end{equation}
 The \begin{it}meromorphic\end{it} and the \begin{it}non-meromorphic parts\end{it} of the elliptic expansion around $\varrho$ are
\begin{align*}
\mathcal{M}_{\varrho}^+ (z)&:=\left(z-\overline{\varrho}\right)^{2k-2}\sum_{n\gg -\infty}c_{\mathcal{M},\varrho}^+(n)
X_{\varrho}(z)^n,\\
\mathcal{M}_{\varrho}^- (z)&:=\left(z-\overline{\varrho}\right)^{2k-2}\sum_{n\ll \infty} c_{\mathcal{M},\varrho}^-(n)\beta_0\left(1-r_{\varrho}(z)^2;2k-1,-n\right) X_{\varrho}(z)^n.
\end{align*}
We refer to the terms in \eqref{eqn:expw} that grow as $z\to \varrho$ as the \begin{it}principal part around $\varrho$\end{it} and denote them by $\ppart_{\mathcal{M},\varrho}$; the corresponding meromorphic and non-meromorphic parts of $\ppart_{\mathcal{M},\varrho}$ are 
\begin{align*}
	\ppart_{\mathcal{M},\varrho}^{+}(z)&:=\left(z-\overline{\varrho}\right)^{2k-2}\sum_{n<0}c_{\mathcal{M},\varrho}^+(n)X_\varrho(z)^n,\\
	\ppart_{\mathcal{M},\varrho}^{-}(z)&:= \left(z-\overline{\varrho}\right)^{2k-2}\sum_{n\geq 0} c_{\mathcal{M},\varrho}^-(n)\beta_0\left(1-r_{\varrho}(z)^2;2k-1,-n\right) X_{\varrho}(z)^n.
\end{align*}

\begin{remark}
Note that the principal parts of the Fourier expansions around all cusps and the principal parts of the elliptic expansions uniquely determine the form.  Indeed, Proposition 3.5 of \cite{BruinierFunke} implies that harmonic Maass forms $\mathcal{M}$ without any singularities must satisfy $\xi_{2-2k}(\mathcal{M})=0$ and there are no non-trivial negative-weight holomorphic modular forms.
\end{remark}

\subsection{Differential operators}\label{sec:diffops}
Recall the raising operator defined in \eqref{raising}.  If $g$ has eigenvalue $\lambda$ and weight $\kappa$, then $R_\kappa^\ell(g)$ ($\ell\in\N_0$) has weight $\kappa+2\ell$ and eigenvalue $\lambda+\kappa\ell+\ell(\ell-1)$. The following lemma may easily be verified by induction on $\ell$.
	\begin{lemma}\label{lem:raiserepeat}
 		For $\ell\in\N_0$ and $g:\H\to\C$ satisfying $\Delta_\kappa(g)=\lambda g$, we have
 		\begin{equation*}
 		R_{-\kappa-2\ell}^{\ell}\left( y^{2\ell+\kappa}\overline{R_\kappa^\ell\left(g(z)\right)}\right) = y^\kappa \prod\limits_{j=1}^{\ell}\left(-\overline{\lambda}-j(j+\kappa-1)\right)\overline{g(z)}.
 		\end{equation*}
 	\end{lemma}
The next lemma rewrites the elliptic coefficients of a  meromorphic function $f$ in terms of the raising operator and $\eta:=\im(\varrho)$. Its proof may be found in Proposition 17 of \cite{BGHZ}.
\begin{lemma}\label{lem:ellexpraise}
	If $f: \mathbb{H}\to\C$ is a meromorphic function that is holomorphic in some neighborhood of $\varrho\in\H$ and $\kappa\in\Z$, then for $z$ in this neighborhood we have
	\[
	f(z)=(2i\eta)^{\kappa} (z-\overline{\varrho})^{-\kappa}\sum_{n\geq 0} \frac{\eta^n}{n!}R_{\kappa}^n(f(\varrho)) X_{\varrho}(z)^{n}.
	\]
\end{lemma}
	We also recall that raising and differentiation are related through Bol's identity ($k\in\N$)
\begin{equation}\label{Bol}
		\mathcal D^{2k-1}=-(4\pi)^{1-2k} R_{2-2k}^{2k-1}.
\end{equation}
We note that the constant $\mathcal{C}_{a,b}$ in \eqref{eqn:beta0def} is chosen so that the operator $\mathcal{D}^{2k-1}$ acts nicely on $\mathcal{M}_{\varrho}^-$.  Namely, most of the terms in $\mathcal{M}_{\varrho}^-$ are annihilated by $\mathcal{D}^{2k-1}$.  One can use this to conclude that $\mathcal{D}^{2k-1}$ maps polar harmonic Maass forms of weight $2-2k$ to meromorphic modular forms of weight $2k$. One can also easily show that $\xi_{2-2k}$ maps polar harmonic Maass forms of weight $2-2k$ to meromorphic modular forms of weight $2k$. This operator may also be written in terms of the raising operator. 
More precisely, for every $g:\mathbb H\rightarrow \C$ we have the equality
\begin{equation}\label{XiR}
\xi_{\kappa}\left(y^{-\kappa}\overline{g(z)}\right)=R_{-\kappa}(g(z)).
\end{equation}

\subsection{Poincar\'e series}

In this section we review the Maass--Poincar\'e series (see Theorem 3.1 of \cite{Fay}) with singularities at the cusps, which are used as inputs of the theta lifts of Theorems \ref{thm:liftfkd} and \ref{thm:Gpolar} (1), and Petersson's meromorphic Poincar\'e series \cite{Pe3,Pe1}, which are closely connected to $f_{\mathcal{A}}$.

To construct the Maass--Poincar\'e series we define, for $Z\in\R\backslash \{0\}$, the expression
\begin{equation*}
\mathcal{M}_{\wt,s}\left(Z\right):=\left|Z\right|^{-\frac{\wt}{2}}M_{\frac{\wt}{2}\sgn(Z),\, s-\frac{1}{2}}\left(|Z|\right),
\end{equation*}
with $M_{\mu,\nu}$ the usual $M$-Whittaker function.   For $\mu,s\in \mathbb{C}$ with $\re\left(s\pm \mu \right)>0$ 
and $Z\in\R^+$ we have
\[
M_{\mu,s-\frac{1}{2}}(Z)=Z^{s}e^{\frac{Z}{2}}\frac{\Gamma(2s)}{\Gamma\left(s+\mu\right)\Gamma\left(s-\mu\right)}\int_0^1 t^{s+\mu-1}(1-t)^{s-\mu-1}e^{-Zt}dt.
\]
For $s= \pm \mu$, we have the well-known identities
\begin{equation}\label{eqn:MWhitSpec}
M_{\mu,\mu-\frac12}(Z)=e^{-\frac{Z}{2}} Z^{\mu}\text{ and } M_{-\mu,\mu-\frac{1}{2}}(Z) = e^{\frac{Z}{2}}Z^{\mu}.
\end{equation}
 For $m\in\Z\setminus\{0\}$ and $\kappa\in\frac{1}{2}\Z$, the function
$$
\psi_{\wt,m,s}\left(\tau\right):=
\left(4\pi |m|\right)^{\frac{\wt}{2}}
\mathcal{M}_{\wt,s}\left(4\pi m v\right)e^{2\pi i mu}
$$
is then an eigenfunction of $\Delta_{\wt}$ with eigenvalue $(s-\frac{\kappa}{2})(1-s-\frac{\kappa}{2})$.  Denoting by $|_{\kappa}\pr$ the identity for $\kappa\in\Z$ and Kohnen's projection operator (see p. 250 of \cite{Kohnen}) for $\kappa\notin\Z$, one concludes that for $\sigma:=\re(s)>1$, the following Poincar\'e series are also eigenfunctions of $\Delta_{\kappa}$ and have weight $\kappa$:
\begin{equation}
\label{eqn:Psdef}
P_{\wt,m,s}:=\sum_{\gamma\in \Gamma_{\infty}\backslash \Gamma_0(4)} \psi_{\wt,\sgn(\wt)m,s}\Big|_{\wt}\gamma\Big|_{\wt}\pr.
\end{equation}
If $s=1-\frac{\kappa}{2}$ or $s=\frac{\kappa}{2}$, then the functions $P_{\wt,m,s}$ are harmonic. 
 The Poincar\'e series satisfy the growth condition
\[
P_{\wt,m,s}(\tau)-\psi_{\wt,m,s}\left(\tau\right)\big|_{\wt}\pr = O\left(v^{1-\re(s)-\frac{\wt}{2}}\right).
\]
Here we simply abbreviate the operator appearing on the right-hand side of the identity on page 250 of \cite{Kohnen} by $\pr$, despite the fact that $\psi_{\wt,m,s}$ is not modular. In the special case that $\wt=k+\frac12>1$ and $s=\frac{\kappa}{2}$, we normalize the resulting weakly holomorphic Poincar\'e series as  
\begin{equation}\label{eqn:gDdef}
P_{k+\frac{1}{2},m}:=\frac{6(4\pi)^{\frac{k}{2}-\frac{1}{4}}}{(k-1)!|m|^{\frac{1}{4}}}P_{k+\frac{1}{2},m,\frac{k}{2}+\frac{1}{4}}.
\end{equation}
For $\wt=\frac32-k<1$ we set
\begin{equation}\label{eqn:PnD3/2-kdef}
\mathcal{P}_{\frac{3}{2}-k,m}:=-\frac{6 (4\pi)^{\frac{k}{2}-\frac{1}{4}} }{(k-1)!|m|^{\frac{1}{4}}(2k-1)}  P_{\frac{3}{2}-k,m,\frac{k}{2}+\frac{1}{4}}.
\end{equation}

We turn next to Poincar\'e series with singularities in the upper half-plane, defining, for $\kappa>2$ even and $n\in\Z$,
\begin{equation}\label{eqn:Psidef}
\Psi_{\kappa,n}(z,\mathfrak{z}):=\sum_{\gamma\in\SL_2 (\Z)}\!\Big(\left(z-\overline{\mathfrak{z}}\right)^{-\kappa}X_{\mathfrak{z}}(z)^n\Big)\bigg|_{\kappa,z}\gamma.
\end{equation}
One has $z\mapsto \Psi_{\kappa,n}(z,\mathfrak{z})\in\SS_{\kappa}$ and $\mathfrak{z}\mapsto\mathbbm{y}^{-\kappa-n}\Psi_{\kappa,n}(z,\mathfrak{z})$ are modular of weight $-2n-\kappa$ (see page 72 of \cite{Pe1}).  Moreover, the functions $z\mapsto \Psi_{\kappa,n}(z,\mathfrak{z})\in\SS_{\kappa}$ vanish identically if $n\not\equiv -\kappa/2\pmod{\omega_{\mathfrak{z}}}$ and are cusp forms in $z$ if $n\in\N_0$. Furthermore, the set $\{\Psi_{\kappa,n}(z,\mathfrak{z}):  \mathfrak{z}\in\H ,n\in\Z\}$ spans $\SS_{\kappa}$ (see S\"atze 7 and 9 of \cite{Pe2}).  The principal part of $z\mapsto \Psi_{\kappa,m}(z,\mathfrak{z})$ has a simple shape. To be more precise, set $f(z)=(2\omega_{\mathfrak{z}})^{-1}\Psi_{\kappa,m}(z,\mathfrak{z})$ and write $c(n):=c_{f,\varrho}(n)-\delta_{n=m}\delta_{\mathfrak{z}=\varrho}$, where in the latter $\delta$-term, and throughout the paper in similar identities, we consider $\mathfrak z$ and $\varrho$ as elements of $\SL_2(\Z)\backslash\H$. Using this notation, Satz 7 of \cite{Pe2} implies that
\begin{equation}\label{eqn:Psiexp}
\left(2\omega_{\mathfrak{z}}\right)^{-1}\Psi_{\kappa,m}\left(z,\mathfrak{z}\right)=\left(z-\overline{\varrho} \right)^{-\kappa}\left(\delta_{\mathfrak{z}=\varrho}X_{\varrho}(z)^m +\sum_{n\geq 0}  c(n) X_{\varrho}(z)^n\right).
\end{equation}
 Moreover, $f_{\mathcal{A}}$ is a specialization of $\Psi_{2k,-k}$, as given in the following straightforward lemma.
\begin{lemma}\label{lem:fPsi}
With $Q_0\in\mathcal{A}\in\mathcal{Q}_{-D}/\SL_2(\Z)$ we have
\begin{equation}\label{eqn:fkDPsi}
f_{\mathcal{A}}(z)=\frac{\left(2v_{Q_0}\right)^{k}}{2\omega_{\tau_{Q_0}}} \Psi_{2k,-k}\left(z,\tau_{Q_0}\right).
 \end{equation}
\end{lemma}

\subsection{Higher Green's functions}\label{sec:Greens}
For $z, \mathfrak{z}\in \mathbb{H}$ and $s\in\mathbb{C}$ with $\sigma>1$, the \textit{automorphic Green's function $G_s$} on $\SL_2(\Z)\backslash\H$ is given by
\begin{equation*}
G_s(z,\mathfrak{z}): =
\sum_{\gamma\in\SL_2(\Z)}g^{\mathbb{H}}_{s}(z,\gamma\mathfrak{z}),
\end{equation*}
where
\begin{equation*}
g^{\mathbb{H}}_{s}(z,\mathfrak{z}):=-\frac{2^{s-1}\,\Gamma(s)^2}{ \Gamma(2s)}
\cosh(d(z,\mathfrak{z}))^{-s}{_2F_1}\left(\frac{s}{2},\frac{s+1}{2};s+\frac{1}{2};\frac{1}{\cosh(d(z,\mathfrak{z}))^2}\right).
\end{equation*}
Note that we have the equality $g^{\mathbb{H}}_{s}(z,\mathfrak{z})=-Q_{s-1}(\cosh(d(z,\mathfrak{z})))$, with $Q_{\nu}$ the associated Legendre function of the second kind.  Furthermore, note that there are different normalizations of $G_s$ in the literature; our normalization agrees with the one from \cite{Mellit}. Automorphic Green's functions can be defined for arbitrary Fuchsian groups of the first kind, and hence in particular for any congruence group. They also arise as the resolvent kernel for the hyperbolic Laplacian (see, e.g.~\cite{Fay, Hejhal}).

In the case $s=k\in\mathbb{N}_{>1}$, the function
$G_k: \mathbb{H}\times\mathbb{H}\to\mathbb{C}$ is called a \textit{higher Green's function}. It is uniquely
characterized by the following properties:
\begin{enumerate}[leftmargin=*]
\item  The function $G_k$ is smooth and real-valued on $\mathbb{H}\times\mathbb{H}\setminus \{(z, \gamma z):\gamma\in\SL_2(\Z), z\in\mathbb{H}\}.$
\item  For $\gamma_1, \gamma_2\in\SL_2(\Z)$ we have $G_k(\gamma_1 z, \gamma_2\mathfrak{z})= G_k(z, \mathfrak{z}).$
\item We have
\[
\Delta_{0, z}\!\left(G_k\!\left(z, \mathfrak{z}\right)\right)=\Delta_{0, \mathfrak{z}}\!\left(G_k\!\left(z, \mathfrak{z}\right)\right)=k(1-k)G_k\left(z, \mathfrak{z}\right).
\]
\item As $z\to\mathfrak{z}$ we have
\[
G_k(z, \mathfrak{z})=2\omega_{\mathfrak{z}}\log\left(r_{\mathfrak{z}}(z)\right)+O(1).
\]
\item As $z$ approaches a cusp, we have $G_k(z, \mathfrak{z})\to 0$.
\end{enumerate}

These higher Green's functions have a long history
(cf. \cite{Fay,GZ,Hejhal,Mellit}).  For example,
Gross and Zagier \cite{GZ} conjectured that their evaluations at CM-points are essentially logarithms of algebraic numbers. If $S_{2k}(\Gamma)=\{0\}$, with $\Gamma\subseteq\SL_2(\Z)$ of finite index, then the conjecture states that 
\[
G_k(z, \mathfrak{z})=(D_1 D_2)^{\frac{1-k}{2}}\log(\alpha)
\]
for CM-points $z, \mathfrak{z}$  of discriminants $D_1 \text{ and } D_2$ respectively and some algebraic number $\alpha$. Various cases of this conjecture have been proved. For example, Mellit \cite{Mellit} proved the case with $k=2$ and $\mathfrak{z}=i$ 
for $\Gamma=\SL_2(\Z)$, and 
also interpreted $\alpha$ as an intersection number of certain algebraic cycles. Further cases were then investigated by Viazovska \cite{Vi}.

\section{Regularized Petersson inner products and the proof of Theorem \ref{thm:innerconverge}}\label{sec:regularization}

\subsection{Known regularized inner products}\label{3.1}
The classical Petersson inner product of two weight $\kappa\in \frac{1}{2}\Z$ (holomorphic) modular forms $f$ and $g$ on $\Gamma_0(N)$ such that $fg$ is a cusp form is given by
\begin{equation}\label{Petint}
\langle f,g \rangle :=
\frac{1}{\left[\SL_2(\Z):\Gamma_0(N)\right]}\int_{\Gamma_0(N)\backslash \H}
 f(z) \overline{g(z)} y^{\kappa}\frac{dx dy}{y^2}.
\end{equation}

Although \eqref{Petint} generally diverges for meromorphic modular forms, one may still define a regularized inner product in some cases.  The first to do so appears to be  Petersson \cite{Pe2}.  If all of the poles of $f$ and $g$ are at the cusps, the regularization of Petersson was later rediscovered and extended by Harvey-Moore \cite{HM} and Borcherds \cite{Bo1}, and subsequently used by Bruinier \cite{Bruinier} and others to obtain a regularized integral that exists in many cases.  Indeed, Petersson gave explicit necessary and sufficient conditions for existence of his regularized inner product in Satz 1a of \cite{Pe2}.  In particular, for $f,g\in\SS_{\kappa}$, his regularization exists if and only if for every $n<0$ and $\varrho \in \mathbb{H}$ we have
\begin{equation}\label{eqn:regexist}
c_{f,\varrho}(n)c_{g,\varrho}(n)=0;
\end{equation}
the conditions are similar if $f$ and $g$ have singularities at the cusps. This regularization is used to define theta lifts of functions with singularities, some of which are evaluated in this paper.

To give a full definition we require some notation.  We let $F_T$ be the restriction of the standard fundamental domain for $\SL_2(\Z)$ to 
those $z$ with $y\leq T$, and let $F_T(N):=\bigcup_{\gamma\in\Gamma_0(N)\backslash \SL_2(\Z)} \gamma F_T.$  For functions $f$ and $g$ transforming like modular forms of weight $\kappa\in \frac{1}{2}\Z$ for $\Gamma_0(N)$ we define
\begin{equation}\label{eqn:innerweaklyPe}
\left<f,g\right> :=\frac{1}{\left[\SL_2(\Z):\Gamma_0(N)\right]} \lim_{T\to\infty}\int_{F_T(N)} f(z)\overline{g(z)}  y^{\kappa}\frac{dx dy}{y^2},
\end{equation}
in the case the integral exists.

The definition above may be interpreted as cutting out neighborhoods around cusps and letting the hyperbolic volume of the neighborhood shrink to zero. If poles exist in $\H$, the construction in \cite{Pe2} is similar.  For $f,g\in\SS_{2k}$ with poles at $\mathfrak{z}_1,\dots,\mathfrak{z}_r\in \SL_2(\Z)\backslash\H$, we choose a fundamental domain $F^*$ such that the representatives of $\mathfrak{z}_1,\dots,\mathfrak{z}_r$ in $F^*$ (also denoted by $\mathfrak{z}_{\ell}$) all lie in the interior of $\Gamma_{\mathfrak{z}_\ell}F^*$. We then set $\mathcal{B}_{\varepsilon}(\mathfrak{z}):=\{ z\in\H : r_{\mathfrak{z}}(z)<\varepsilon\}$, and define Petersson's regularized inner product as
\begin{equation}\label{eqn:innermeroPe}
\langle f,g \rangle := \lim_{\varepsilon_1,\dots,\varepsilon_r\to 0^+} \int_{F^*{\backslash} \bigcup_{\ell=1}^r \mathcal{B}_{\varepsilon_{\ell}}\left(\mathfrak{z}_\ell\right)} f(z) \overline{g(z)} y^{2k}\frac{dx dy}{y^2}.
\end{equation}
Like \eqref{eqn:innerweaklyPe}, the regularized inner product \eqref{eqn:innermeroPe} does not always exist.  The inner product \eqref{eqn:innerweaklyPe} has recently been further extended to an inner product on all weakly holomorphic modular forms by Diamantis, Ehlen, and the first author in \cite{BDE}, and we address the extension of \eqref{eqn:innermeroPe} to all meromorphic cusp forms in the next section.

\subsection{A new regularization}
In this section we restrict to $\kappa=2k\in 2\Z$ and $N=1$, but the construction can be easily generalized to subgroups.  We also assume that $f$ and $g$ decay like cusp forms towards the cusps, but this restriction can be removed by combining with the regularization from Subsection \ref{3.1}. We choose a fundamental domain $F^*$ as in Subsection \ref{3.1} and denote the poles of $f$ and $g$ in $F^*$ by $\mathfrak{z}_1,\dots,\mathfrak{z}_{r}$.

For an analytic function $A(s)$ in $s=(s_1,\dots,s_r)$, denote by $\mathrm{CT}_{s=0}A(s)$  the constant term of the meromorphic continuation of $A(s)$ around $s_1=\cdots =s_{r}=0$, and define
\begin{equation}\label{eqn:OurReg}
\left<f,g\right>:= \operatorname{CT}_{s=0}\left(\int_{\SL_2(\Z)\backslash \H} f(z) H_s(z) \overline{g(z)} y^{2k}\frac{dx dy}{y^2}\right)
\end{equation}
where
\[
H_s (z) = H_{s_1,\dots, s_r, \mathfrak{z}_1,\dots,\mathfrak{z}_r} (z) := \prod_{\ell=1}^r h_{s_{\ell},\mathfrak{z}_\ell} (z).
\]
Here for $\mathfrak{z}_\ell\in F^*$ and $z\in\mathbb{H}$ we set $h_{s_{\ell},\mathfrak{z}_\ell}(z):=r_{\mathfrak{z}_\ell}(\gamma z)^{2s_{\ell}}$, with $\gamma \in \SL_2(\Z)$ chosen such that $\gamma z\in F^*$.  Note that  $r_{\mathfrak{z}_\ell}(\gamma z)\to 0$ as $z\to \gamma^{-1}\mathfrak{z}_\ell$, so the integral in \eqref{eqn:OurReg} converges for $\sigma\gg 0$, where this notation means that for every $1\leq \ell\leq r$, $\sigma_{\ell}:=\re(s_{\ell})\gg 0$.  One can show that the regularization is independent of the choice of fundamental domain.

\begin{proof}[Proof of Theorem \ref{thm:innerconverge}]

For $\delta>0$ sufficiently small, we may assume that the $\mathcal{B}_{\delta}(\mathfrak{z}_\ell)$ are disjoint and split off the integral over those $z$ that lie in one of these balls.

If $z\notin \mathcal{B}_{\delta}(\mathfrak{z}_\ell)$ for all $\ell$, then one can bound the integrand locally uniformly for $s$ contained in a small open neighborhood around $0$.  Hence we conclude that
\begin{equation}\label{eqn:largereval}
\operatorname{CT}_{s=0}\left(\int_{F^*\backslash\bigcup_{\ell=1}^{r} \mathcal{B}_{\delta}\left(\mathfrak{z}_\ell\right)} f(z)H_s(z) \overline{g(z)} y^{2k}\frac{ dx dy}{y^2}\right)=\int_{F^* \backslash  \bigcup_{\ell=1}^{r} \mathcal{B}_{\delta}\left(\mathfrak{z}_\ell\right)} f(z)\overline{g(z)}y^{2k}\frac{ dx dy}{y^2}.
\end{equation}
Thus we are left to show existence of the meromorphic continuation to a small open neighborhood around $0$ of 
\begin{equation}\label{eqn:ballint}
\int_{\mathcal{B}_{\delta}(\mathfrak{z}_\ell)\cap F^*} f(z) H_s(z) \overline{g(z)} y^{2k}\frac{dx dy}{y^2}.
\end{equation}
By construction, $\mathfrak{z}_\ell$ lies in the interior of $\Gamma_{\mathfrak{z}_\ell}F^*$, so we may assume that $\mathcal{B}_{\delta}(\mathfrak{z}_\ell)\subseteq \Gamma_{\mathfrak{z}_\ell}F^*$.
To rewrite \eqref{eqn:ballint} as an integral over the entire ball $\mathcal{B}_{\delta}(\mathfrak{z}_\ell)$, we decompose 
the ball 
into the disjoint union
\begin{equation}\label{eqn:Ballsplit}
\mathcal{B}_{\delta}(\mathfrak{z}_\ell) =\overset{\bullet}{\bigcup}_{\gamma\in \Gamma_{\mathfrak{z}_\ell}} \gamma\left( \mathcal{B}_{\delta}\left(\mathfrak{z}_\ell\right)\cap F^*\right).
\end{equation}
Moreover, bounding $h_{s_m,\mathfrak{z}_{m}}(z)$ locally uniformly for $\sigma_m>-\varepsilon$ for $m\neq \ell$, we may plug in $s_m=0$, and hence the invariance of the integrand under $\SL_2(\Z)$ implies that the constant term at $s=0$ of \eqref{eqn:ballint} is the constant term at $s_{\ell}=0$ of
$$
\mathcal{I}_{s_{\ell},\mathfrak{z}_\ell,\delta}(f,g):= \frac{1}{\omega_{\mathfrak{z}_\ell}}\int_{\mathcal{B}_{\delta}(\mathfrak{z}_\ell)} f(z) h_{s_{\ell},\mathfrak{z}_\ell}(z) \overline{g(z)} y^{2k}\frac{dx dy}{y^2}.
$$
Setting $R:=r_{\mathfrak{z}_\ell}(z)$ for $z\in \mathcal{B}_{\delta}(\mathfrak{z}_\ell)$, we have for $\gamma\in \Gamma_{\mathfrak{z}_\ell}$ the equality
\begin{equation}\label{eqn:requal}
h_{s_{\ell},\mathfrak{z}_\ell}(z)=r_{\mathfrak{z}_\ell}(\gamma z)^{2s_{\ell}}=R^{2s_{\ell}}.
\end{equation}
We now closely follow the proof of Satz 1a in \cite{Pe2}. We rewrite
\[
y=\frac{\mathbbm{y}_{\ell}\left(1-r_{\mathfrak{z}_{\ell}}^2(z)\right)}{\left|1-X_{\mathfrak{z}_{\ell}}(z)\right|^2}
\]
for $\mathfrak{z}_{\ell}=\mathbbm{x}_\ell+i\mathbbm{y}_\ell$ and compute
\[
dz=\frac{2iy}{\left(1-X_{\mathfrak{z}_{\ell}}(z)\right)^2} dX_{\mathfrak{z}_{\ell}}(z).
\]
Hence changing variables $X_{\mathfrak{z}_\ell}(z) = Re^{i\vartheta}$ and inserting the elliptic expansions \eqref{eqn:fEllExp} of $f, g$ around $\varrho=\mathfrak{z}_\ell$ yields
\begin{align}
\nonumber\mathcal{I}_{s_{\ell},\mathfrak{z}_\ell,\delta}(f,g)&=
\frac{1}{\omega_{\mathfrak{z}_\ell}} \int_{\mathcal{B}_\delta(\mathfrak{z}_\ell)}  f(z) \overline{g(z)}r_{\mathfrak{z}_\ell}(z)^{2s_{\ell}}  y^{2k} \frac{dx dy}{y^2}\\
\nonumber &=\frac{4}{\omega_{\mathfrak{z}_\ell}\left(4 \mathbbm{y}_\ell\right)^{2k}}\sum_{m,n\gg -\infty}
c_{f,\mathfrak{z}_\ell}(n)\overline{c_{g,\mathfrak{z}_\ell}(m)}\int_{0}^{\delta} \int_{0}^{2\pi} e^{i(n-m)\vartheta}R^{n+m+2s_{\ell}}\left(1-R^2\right)^{2k-2} R d\vartheta dR\\
&=\frac{8\pi}{\omega_{\mathfrak{z}_\ell}\left(4\mathbbm{y}_\ell\right)^{2k}}\sum_{n\gg -\infty}
c_{f,\mathfrak{z}_\ell}(n)\overline{c_{g,\mathfrak{z}_\ell}(n)} \int_{0}^{\delta} R^{1+2n+2s_{\ell}}\left(1-R^2\right)^{2k-2} dR.\label{eqn:innerconverge}
\end{align}
Plugging in the binomial expansion
of $(1-R^2)^{2k-2}$, the
 remaining integral in \eqref{eqn:innerconverge} becomes
\begin{equation*}
\sum_{j=0}^{2k-2} (-1)^j \binom{2k-2}{j} \int_{0}^{\delta} R^{1+2(n+j)+2s_{\ell}} dR = \sum_{j=0}^{2k-2} (-1)^j \binom{2k-2}{j} \frac{\delta^{2\left(n+j+1+s_{\ell}\right)}}{2\left(n+j+1+s_{\ell}\right)}.
\end{equation*}
Since this is meromorphic at $s_{\ell}=0$, its constant term at $s_{\ell}=0$ exists, yielding the existence of the inner product. 

We next prove that the inner product is Hermitian. For $f,g\in \SS_{2k}$, let $F_{f,g}$ denote the meromorphic continuation of the function defined for $s\in\C^r$ with $\sigma_{\ell}\gg 0$ by
\[
F_{f,g}(s):=\int_{ F^*} f(z)\overline{g(z)} H_s(z) y^{2k} \frac{dxdy}{y^2}.
\]
Since $\langle f, g\rangle$ always exists, $F_{f,g}$ has an expansion around $s=0$ of the shape
\[
F_{f,g}(s)=\sum_{n=(n_1,\ldots,n_r)\in\Z^r} a_{f,g}(n) s_1^{n_1}\cdot\cdots\cdot s_r^{n_r}
\]
with $\langle f,g\rangle=a_{f,g}(0)$.  Since $r_{\mathfrak{z}_\ell}(z)\in\R$, we have
$\overline{H_{\overline{s}}(z)}=H_{s}(z)$,
and thus
\[
\overline{\langle g,f\rangle}=\overline{a_{g,f}(0)}=\operatorname{CT}_{s=0}\!\left(\overline{F_{g,f}\left(\overline{s}\right)}\right)=\operatorname{CT}_{s=0}\!\left(\int_{ F^*} f(z)\overline{g(z)}\ \overline{H_{\overline{s}}(z)} y^{2k} \frac{dxdy}{y^2}\right)
=\langle f,g\rangle.
\]

We finally show that the new regularization agrees with Petersson's, wherever his exists. Setting $\mathcal{B}\left( \mathfrak{z}_\ell,\varepsilon,\delta\right):=\left\{ z\in\, \H\, : \varepsilon<r_{\mathfrak{z}_\ell}(z)<\delta\right\},$ Petersson's regularization equals
\begin{multline}\label{eqn:Petreg}
\lim_{\varepsilon_1,\dots,\varepsilon_r\to 0^+} \int_{ F^*{\backslash}\bigcup_{\ell=1}^r \mathcal{B}_{\varepsilon_{\ell}}\left(\mathfrak{z}_\ell\right)} f(z) \overline{g(z)} y^{2k}\frac{dx dy}{y^2}\\
=\int_{ F^*{\backslash} \bigcup_{\ell=1}^{r} \mathcal{B}_{\delta}\left(\mathfrak{z}_\ell\right)} f(z) \overline{g(z)} y^{2k}\frac{dx dy}{y^2} +\lim_{\varepsilon_1,\dots,\varepsilon_r\to 0^+} \int_{ F^*\cap\bigcup_{\ell=1}^{r} \mathcal{B}\left(\mathfrak{z}_\ell,\varepsilon_{\ell},\delta\right)} f(z) \overline{g(z)} y^{2k}\frac{dx dy}{y^2}.
\end{multline}
The first term on the right-hand side of \eqref{eqn:Petreg} is precisely the right-hand side of \eqref{eqn:largereval}.  It thus remains to prove that
\begin{equation}\label{eqn:limball}
\lim_{\varepsilon_{\ell}\to 0^+} \int_{  F^{*}\cap\mathcal{B}\left(\mathfrak{z}_\ell,\varepsilon_{\ell},\delta\right)}f(z) \overline{g(z)} y^{2k}\frac{dx dy}{y^2}
=\operatorname{CT}_{s_{\ell}=0}\mathcal{I}_{s_{\ell},\mathfrak{z}_\ell,\delta}(f,g).
\end{equation}
By the existence condition \eqref{eqn:regexist}, we may plug in $s_{\ell}=0$ in \eqref{eqn:innerconverge}, and obtain that
$$
\operatorname{CT}_{s_{\ell}=0}\mathcal{I}_{s_{\ell},\mathfrak{z}_\ell,\delta}(f,g)=\frac{8\pi}{\omega_{\mathfrak{z}_\ell}\left(4\mathbbm{y}_\ell\right)^{2k}}\sum_{n\geq 0}
c_{f,\mathfrak{z}_\ell}(n)\overline{c_{g,\mathfrak{z}_\ell}(n)}
\int_{0}^{\delta} R^{1+2n}\left(1-R^2\right)^{2k-2} dR.$$
Using \eqref{eqn:Ballsplit} and then following the calculation in \eqref{eqn:innerconverge}, plugging the elliptic expansion into the left-hand side of \eqref{eqn:limball} yields that the two regularizations match since
\[
\lim_{\varepsilon_{\ell}\to 0^+} \int_{\mathcal{B}\left(\mathfrak{z}_\ell,\varepsilon_{\ell},\delta\right)}f(z) \overline{g(z)} y^{2k} \frac{dx dy}{y^2}
= \frac{8\pi}{\left(4\mathbbm{y}_\ell\right)^{2k}}\sum_{n\geq 0}
c_{f,\mathfrak{z}_\ell}(n)\overline{c_{g,\mathfrak{z}_\ell}(n)}
\int_{0}^{\delta} R^{1+2n}\left(1-R^2\right)^{2k-2}dR.
\]
\end{proof}

\section{Theta lifts and the proofs of Theorems \ref{thm:liftfkd} and \ref{thm:Gpolar} (1)}\label{sec:theta}

\subsection{Proof of Theorem \ref{thm:liftfkd}}

Before proving Theorem \ref{thm:liftfkd}, we note a (well-known) fact about the theta kernel $\Theta_k$.
\begin{lemma}
The function
 $\tau\mapsto \Theta_k(z,\tau)$ grows at most polynomially towards the cusps and decays exponentially for $\tau\to i\infty$.
\end{lemma}
\begin{proof}
  To show that it exponentially decays towards $i\infty$, we use \eqref{eqn:Qrewrite} to rewrite the absolute value of the exponential in definition \eqref{eqn:thetadef1} as $e^{-2\pi v (Q_z^2 +|Q(z,1)|^2/y^2)}$, and then note that $Q_z^2 +|Q(z,1)|^2/y^2>0$ for $Q\neq [0,0,0]$.
Modularity then implies the claim at the other cusps.
\end{proof}

We specifically apply the theta lift $\Phi_k$ to the weight $k+\frac12$ weakly holomorphic Poincar\'e series $P_{k+\frac12,-D}$ defined in \eqref{eqn:gDdef} in order to obtain Theorem \ref{thm:liftfkd}.
\begin{proof}[Proof of Theorem \ref{thm:liftfkd}]
A standard unfolding argument (see, e.g., \cite{Zagiernotrapid}) combined with \eqref{eqn:MWhitSpec} gives that $\Phi_k(P_{k+\frac12,-D,\frac{k}{2}+\frac14})$ equals 
\begin{multline}\label{Zagiersplit}
\frac{(4\pi D)^{\frac{k}{2}+\frac{1}{4}}}{6}\lim_{T\to\infty}\left(\int_0^{T} \int_0^1 e^{-2\pi iD\tau}\overline{\Theta_k (z, \tau)}v^{k+\frac{1}{2}}\frac{dudv}{v^2}\vphantom{-\sum_{c\geq 1}\sum_{a\!\!\pmod{c}^*}\int_{S_{\frac{a}{c}}} e^{-2\pi iD\tau}\overline{\Theta_k (z, \tau)}v^{k+\frac{1}{2}}\frac{dudv}{v^2}}\right.
\\
\left.\vphantom{\int_0^{T} \int_0^1 e^{-2\pi iD\tau}\overline{\Theta_k (z, \tau)}v^{k+\frac{1}{2}}\frac{dudv}{v^2}}
-\sum_{c\geq 1}\sum_{a\!\!\pmod{c}^*}\int_{S_{\frac{a}{c}}} e^{-2\pi iD\tau}\overline{\Theta_k (z, \tau)}v^{k+\frac{1}{2}}\frac{dudv}{v^2}\right),
\end{multline}
where $a$ runs over residues modulo $c$ that are coprime to $c$ and for each $a$ and $c$ we denote by $S_{\frac{a}{c}}$ the disc of radius $(2c^2T)^{-1}$ tangent to the real axis at $\frac{a}{c}$. 
Note that the factor $\frac16=[\text{SL}_2(\mathbb Z):\Gamma_0(4)]$ comes from the fact that the inner product is taken over $\Gamma_0(4)$. 
Following an argument similar to the proof of Theorem 1.1 (2) in \cite{BKM}, 
the polynomial growth of $\tau\mapsto\Theta_k (z,\tau)$ towards the cusps yields that the second term of \eqref{Zagiersplit} does not contribute in the limit $T\to\infty$. To evaluate the integral in the first term of \eqref{Zagiersplit}, we plug in the defining series \eqref{eqn:thetadef1} and integrate over $u$ to obtain, as $T\rightarrow\infty$, the expression
\[
\frac{(4\pi D)^{\frac{k}{2}+\frac{1}{4}}}{6}y^{-2k}\sum_{Q\in\mathcal{Q}_{-D}}Q\left(\overline{z},1\right)^{k}\int_0^{\infty} e^{4\pi Dv-4 \pi Q_z^2 v}v^{k-1}dv.
\]
The claim now easily follows using \eqref{eqn:Qrewrite} to show that the integral on $v$ equals 
$$
\int_0^{\infty}e^{-\frac{4\pi |Q(z,1)|^2 v}{y^2}}v^{k-1}dv =\frac{(k-1)!}{(4\pi)^k}
y^{2k}|Q(z,1)|^{-2k}.
$$
\end{proof}

\subsection{Proof of Theorem \ref{thm:Gpolar} (1)}
The goal of this section is to compute the image of the Maass--Poincar\'e series $P_{\frac32-k,-D,s}$ defined in \eqref{eqn:Psdef} under $\Phi_{1-k}^*$, and to connect these images to the functions $\mathcal F_{\mathcal{A}}$.  We do so in the following theorem, which extends Theorem \ref{thm:Gpolar} (1).
\begin{theorem}\label{thm:thetalift}
For $s\in\C$ with $\sigma>1$ we have
\begin{align}\notag
&\Phi_{1-k}^*\left(P_{\frac{3}{2}-k,-D,s}\right)(z)\\
\label{eqn:Phivals}
&\qquad\qquad=\frac{D^s\,\Gamma\left(s+\frac{k}{2}-\frac14\right)}{6(4\pi)^{\frac{k}{2}-\frac14}}\sum_{Q\in \mathcal{Q}_{-D}} Q_z^{-2s-k+\frac{3}{2}} Q(z,1)^{k-1}{_2F_1}\left(s+\frac{k}{2}-\frac{1}{4},s+\frac{k}{2}-\frac{3}{4}; 2s; \frac{D}{Q_z^2}\right).
\end{align}
In particular, we have the equality
\[
\Phi_{1-k}^* \left( \mathcal{P}_{\frac{3}{2}-k,-D}\right) =\mathcal{F}_{1-k,-D}.
\]
\end{theorem}
\begin{remark} 
 After seeing a preliminary version of this paper, Zemel \cite{Zemel} has obtained further theta lifts related to vector-valued versions of $\mathcal{F}_{\mathcal{A}}$. 
\end{remark}
\begin{proof}[Proof of Theorem \ref{thm:thetalift}]
One can show that both sides of \eqref{eqn:Phivals} converge absolutely and locally uniformly for $\sigma>1$ and $
z\notin\{ \tau_Q: Q\in \mathcal{A}\}
$ and hence are analytic in $s$. It thus suffices to show \eqref{eqn:Phivals} for $\sigma>k-\frac32$.
Following the proof of Theorem \ref{thm:liftfkd} 
and noting that the map $[a,b,c]\mapsto [a,-b,c]$ is an involution on $\mathcal{Q}_{-D}$, we obtain 
\begin{equation}\label{eqn:Phi*eval}
\Phi_{1-k}^*\left(P_{\frac{3}{2}-k,-D,s}\right)(z)=
\frac{(4\pi D)^{\frac{1}{4}-\frac{k}{2}}}{6} \sum_{Q\in\mathcal{Q}_{-D}} Q_z Q(z,1)^{k-1}\mathcal{I}_s\left(\frac{Dy^2}{|Q(z,1)|^2}\right),
\end{equation}
where 
$$
\mathcal{I}_s(Z):=\int_0^{\infty}\mathcal{M}_{\frac{3}{2}-k,s}\left(v\right) v^{-\frac{1}{2}}e^{-\frac{v}{2}} e^{-\frac{v}{Z}} dv\quad\left(Z\in\R^+\right).
$$
Applying formula 7.621.1.~of \cite{GR} yields
\begin{equation}\label{eqn:Ieval}
\mathcal{I}_{s}(Z) =\Gamma\left(s+\frac{k}{2}-\frac{1}{4}\right)\left( \frac{Z}{Z+1}\right)^{s+\frac{k}{2}-\frac{1}{4}}{_2F_1}\left(s+\frac{k}{2}-\frac{1}{4}, s+\frac{k}{2}-\frac{3}{4}; 2s; \frac{Z}{Z+1}\right).
\end{equation}
Moreover, \eqref{eqn:Qrewrite} implies that 
\begin{equation}\label{eqn:rel-w-Qz}
\frac{Z}{Z+1}= \frac{D}{Q_z^2}
\end{equation}
for $Z:=Dy^2/|Q(z,1)|^2$, and substituting this and \eqref{eqn:Ieval} back into \eqref{eqn:Phi*eval} yields \eqref{eqn:Phivals}.

To prove the second claim, we use \eqref{2F1tr}, \eqref{2F1sym}, and \eqref{equ-beta-hypergeom} to obtain 
\begin{align}\nonumber
 	{_2F_1}\left(k,k-\frac{1}{2};k+\frac{1}{2}; W\right)
	&=
\nonumber
	(1-W)^{\frac{1}{2}-k}\,{_2F_1}\!\left(\frac{1}{2},k-\frac{1}{2};k+\frac{1}{2};\frac{W}{W-1}\right)\\
\label{eqn:NISTbeta}	&=(-1)^{k-\frac{1}{2}}\left(k-\frac{1}{2}\right)W^{-k+\frac{1}{2}}\,\beta\!\left(\frac{W}{W-1};k-\frac{1}{2},\frac{1}{2}\right).
 	\end{align}
Employing \eqref{eqn:NISTbeta} with $W:=D/Q_z^2$ and using \eqref{eqn:rel-w-Qz} to evaluate $W/(W-1)=-Z=-Dy^2/|Q(z,1)|^2$, we obtain the second claim, using the fact that $\sgn(Q_z)=1$.
\end{proof}

\section{Properties of the function $\mathcal{F}_{\mathcal{A}}$ and proof of theorem \ref{thm:Gpolar} (2)}\label{sec:FA}
\subsection{Relation to higher Green's functions}
The goal of this section is to write
the functions
 $\mathcal{F}_{\mathcal A}$ defined in \eqref{eqn:Gdef} in terms of higher Green's functions.  
 For this, set   
 \begin{align}\label{eqn:akndef}
 a_{k,n}:=
-\frac{(2k-2)!}{2^{k-1}(k-1)!}
 \begin{cases}1 & \text{if }n\geq k-1,\\
 \frac{n!}{(2k-2-n)!}& \text{if }n<k-1.
 \end{cases}
 \end{align}

 \begin{lemma}\label{lem:diffopseed}
 	For $Q\in\mathcal{Q}_{-D}$ and $n\in\N_0$ we have
 	\[
 	R_{2-2k}^n\!\left(\mathbb{P}_{1-k,-D,Q}(z)\right) = a_{k,n}
 	\begin{cases} R_0^{n+1-k}\!\left(g_k^{\H}(z,\tau_{Q})\right) & \text{ if }n\geq k-1,\\[0.1cm]
 	 y^{2k-2-2n} \overline{R_{0}^{k-1-n}\!\left(g_k^{\H}(z,\tau_{Q})\right)} & \text{ if }
n\leq  k-1.
 	\end{cases}
 	\]
 \end{lemma}

 \begin{proof}
We first prove that for 
every $z,\mathfrak z\in\HH$ and $j\in\N_0$ we
 have
\begin{multline}\label{greenraise}
R_{0,z}^{j}\!\left(g^{\mathbb{H}}_{k}(z,\mathfrak{z})\right)
\\
=\frac{-2^{k-j-1}(k-1)!(k+j-1)!(\overline{z}-\mathfrak{z})^j(\overline{z}-\overline{\mathfrak{z}})^j}{(2k-1)!y^{2j}\mathbbm{y}^j\cosh(d(z,\mathfrak{z}))^{k+j}}\,{_2F_1}\left(\frac{k+j+1}{2},\frac{k+j}{2};k+\frac{1}{2};\frac{1}{\cosh(d(z,\mathfrak{z}))^2}\right).
\end{multline}
 	We note that the images under repeated raising of the Green's function are known in the literature (see, e.g., \cite{Mellit}), but some rewriting is still required to derive the form \eqref{greenraise}. So, for the convenience of the reader, we present a direct proof. To show \eqref{greenraise}, we first compute 
\begin{equation}\label{raisecos}
R_{0,z}\left(\cosh(d(z,\mathfrak{z}))\right)=-\frac{(\overline{z}-\mathfrak z)(\overline{z}-\overline{\mathfrak z})}{2y^2\mathbbm{y}},
\end{equation}
 from which we conclude that $R^{2}_{0,z}\left(\cosh(d(z,\mathfrak{z}))\right)=0$.  Employing these identities, induction on $j\in\mathbb{N}_0$ gives 
\begin{multline}\label{raisefirst}
R_{0,z}^{j}\!\left(g^{\mathbb{H}}_{k}(z,\mathfrak{z})\right)\\
=-\frac{2^{k-1}\,(k-1)!^2}{ (2k-1)!}\, \left(R_{0,z}\!\left(\cosh(d(z,\mathfrak{z}))\right) \right)^j\frac{\partial^j}{\partial Z^j}\left[Z^{-k}\,{_2F_1}\left(\frac{k}{2},\frac{k+1}{2};k+\frac{1}{2};\frac{1}{Z^2}\right)\right]_{Z=\cosh(d(z,\mathfrak{z}))}.
\end{multline}
 		Next, again by induction on $j\in\mathbb{N}_0$, and employing \eqref{diff2F1}, we obtain
 		\begin{align*}
 		\frac{\partial^j}{\partial Z^j}\left(Z^{-k}\,{_2F_1}\left(\frac{k}{2},\frac{k+1}{2};k+\frac{1}{2};\frac{1}{Z^2}\right)\right)
 		=\frac{(-1)^{j}(k+j-1)!}{(k-1)!Z^{k+j}}
		\,{_2F_1}\left(\frac{k+j+1}{2},\frac{k+j}{2};k+\frac{1}{2};\frac{1}{Z^2}\right).
 		\end{align*}
 		Plugging this and \eqref{raisecos} into \eqref{raisefirst} gives \eqref{greenraise}.
		
	Using \eqref{greenraise}, we next show the $n=0$ case of the assertion of Lemma \ref{lem:diffopseed}, namely
 	\begin{equation}\label{n0}
 	\mathbb{P}_{1-k,-D,Q}(z) = -\frac{1}{2^{k-1}(k-1)!} y^{2k-2}\overline{R_{0}^{k-1}\left(g_k^{\H}(z,\tau_Q)\right)}.
 	\end{equation}
 	For this we let $j=k-1$ in \eqref{greenraise}, which yields that
 	\begin{equation}\label{jk1}
	 	R_{0,z}^{k-1}\!\(g^{\mathbb{H}}_{k}(z,\mathfrak{z})\right)=-\frac{(k-1)!(\overline{z}-\mathfrak{z})^{k-1}(\overline{z}-\overline{\mathfrak{z}})^{k-1}}{(2k-1)y^{2k-2}\mathbbm{y}^{k-1}\cosh(d(z,\mathfrak{z}))^{2k-1}}{_2F_1}\!\left(k,k-\frac{1}{2};k+\frac{1}{2};\frac{1}{\cosh(d(z,\mathfrak{z}))^2}\right).
 	\end{equation}
From now on we choose $\mathfrak{z}=\tau_Q$. Hence, employing \eqref{eqn:NISTbeta}, \eqref{eqn:coshrat}, and \eqref{rewriteQ}, 
\eqref{jk1} becomes
 	\begin{align*}
 	R_{0}^{k-1}\!\left(g^{\mathbb{H}}_{k}\!\left(z,\tau_{Q}\right)\right)=
 	i(-1)^{k}2^{k-2} (k-1)! D^{\frac{1-k}{2}} y^{2-2k} \overline{Q(z,1)}^{k-1}\beta\!\left( -\frac{Dy^2}{\abs{Q(z,1)}^2}; k-\frac{1}{2},\frac{1}{2}\right).
 	\end{align*}
 	Since $Q$ is positive-definite, so that $\sgn(Q_z)=1$, this gives \eqref{n0}.
 	 	
 	To finish the proof, we apply raising $n$ times to \eqref{n0}, yielding
 	\begin{equation}\label{eqn:raiseseeds}
 	R_{2-2k}^n\left(\mathbb{P}_{1-k,-D,Q}(z)\right)=-\frac{1}{2^{k-1}(k-1)!}R_{2-2k}^n\left(y^{2k-2}\overline{R_0^{k-1}\left(g_k^{\H}(z,\tau_Q)\right)}\right).
 	\end{equation}
 	We now distinguish among two cases depending on whether $n\geq k-1$ or $n\leq k-1$.

In the case $n\geq k-1$, the claim follows from applying Lemma \ref{lem:raiserepeat} with $\ell=k-1$ to \eqref{eqn:raiseseeds} and noting that $g_k^{\H}$ is real-valued. 	
This yields the claim.

For  $n\leq k-1$ the eigenvalue of $R_{0}^{k-1-n}(g_k^{\H})$ is
 	$ (n+1)(n+2-2k).$
 	Thus we have, by Lemma \ref{lem:raiserepeat} with $\ell=n$,
 	\begin{align*}
 	R_{2-2k}^n\!\left(y^{2k-2}\overline{R_{2k-2-2n}^n\left(R_{0}^{k-1-n}\!\left(g_k^{\H}\!\left(z,
\tau_Q
\right)\right)\right)}\right)	
=  y^{2k-2-2n}\frac{n!(2k-2)!}{(2k-2-n)!} \overline{R_0^{k-1-n}\!\left(g_k^{\H}\!\left(z,
\tau_Q
\right)\right)}.
 	\end{align*}
 	Plugging this back into \eqref{eqn:raiseseeds} yields the claim.
 \end{proof}
 We also need regularized versions of $G_k$ and $\mathcal F_{\mathcal{A}}$. For this, define
\[
 \sideset{}{^{\operatorname{reg}}}\sum\limits_{w\in S} h(w):=\sum_{\substack{w\in S\\ h(w)\neq \infty}} h(w),
\]
 where $h$ is an arbitrary function taking inputs from some set $S$ and with outputs in $\C\cup \{\infty\}$. Note that different choices of $w$ lead to different subsets of $S$ being excluded on the right-hand side of the above equation. For any operator $\mathcal O$ we then let
 $$
 \mathcal O\left(\sideset{}{^{\operatorname{reg}}}\sum\limits_{w\in S} h(w)\right):=\sideset{}{^{\operatorname{reg}}}\sum\limits_{w\in S} \mathcal O(h(w)).
 $$
Moreover, for $\mathcal{H}(z):=\sum_{w\in S} h_z(w)$ we set
\[
\mathcal{H}^{\operatorname{reg}}(z):=\sideset{}{^{\operatorname{reg}}}{\sum}_{w\in S} h_z(w).
\]
  Note that $\mathcal{H}$ may possibly have distinct presentations of this type, written both as a sum over $w\in S$ of $h_z(w)$ and as sum over another set with a different function. The regularization $\mathcal{H}^{\operatorname{reg}}(z)$ depends on the choice of its presentation so we emphasize that the regularization uses the choice of $h_z$ and the set $S$ given in the definition of $\mathcal{H}$.

  We obtain the following corollary by applying Lemma \ref{lem:diffopseed} termwise. 

 \begin{corollary}\label{diffop}
 For $Q_0\in\mathcal{A}\in\mathcal{Q}_{-D}/\SL_2(\Z)$ and $n\in\N_0$ we have
 	\[
 	R_{2-2k}^n\!\left(\mathcal F_{\mathcal{A}}^{\operatorname{reg}}(z)\right) = \frac{a_{k,n}}{2\omega_{\tau_{Q_0}}} \begin{cases} R_0^{n+1-k}\!\left(G_k^{\operatorname{reg}}(z,\tau_{Q_0})\right) & \text{ if }n\geq k-1,\\ y^{2k-2-2n} \overline{R_{0}^{k-1-n}\!\left(G_k^{\operatorname{reg}}(z,\tau_{Q_0})\right)} & \text{ if }
n\leq  k-1,
 	\end{cases}
 	\]
where $a_{k,n}$ is defined in \eqref{eqn:akndef}.
In particular 
\[
\mathcal{F}^{\mathrm{reg}}_{\mathcal{A}}(z)=-\frac1{2^k
(k-1)!\,\omega_{\tau_{Q_0}}}y^{2k-2}\overline{
R_0^{k-1}\left(G_k^{\mathrm{reg}}\left(z, \tau_Q\right)\right)}.
\]
 \end{corollary}

 \begin{proof}
 	Writing $Q=Q_0\circ M$ with $M\in \Gamma_{\tau_{Q_0}}\backslash\SL_2(\Z)$, we have
 	\[
 	R_{2-2k}^n\!\left(\mathcal F_{\mathcal{A}}^{\operatorname{reg}}\right)=\frac{1}{2\omega_{Q_0}}\ \ \  \sideset{}{^{\operatorname{reg}}}\sum_{M\in\SL_2(\Z)} R_{2-2k}^{n}\!\left(\mathbb{P}_{1-k,-D,Q_0\circ M}\right).
 	\]
 	The result then follows from Lemma \ref{lem:diffopseed}, using the relation $\tau_{Q_0\circ M} =M^{-1} \tau_{Q_0}$.
 \end{proof}

\subsection{Proof of Theorem \ref{thm:Gpolar} (2)}\label{sec:GQprop}

We now have the necessary pieces to show the modularity of $\mathcal{F}_{\mathcal{A}}$ and its relation to $f_{\mathcal{A}}$ under the differential operators.
\begin{proof}[Proof of Theorem \ref{thm:Gpolar} (2)]
We first show that
\begin{equation}\label{eqn:RkGreens}
	R_0^k\left(G_k\!\left(z,\tau_{Q_0}\right)\right)= -2^{k}(k-1)!\,\omega_{\tau_{Q_0}} f_{\mathcal{A}}(z).
	\end{equation}
For this, we plug $j=k$ into \eqref{greenraise}, which implies that $R_0^k\left(g_k^{\H}(z,\mathfrak{z})\right)$ equals
	\begin{align}
	-\frac{(k-1)!(\overline{z}-\mathfrak{z})^k(\overline{z}-\overline{\mathfrak{z}})^{k}}{2y^{2k}\mathbbm{y}^k\cosh(d(z,\mathfrak{z}))^{2k}} {_2F_1}\left(k+\frac{1}{2},k;k+\frac{1}{2};\frac{1}{\cosh(d(z,\mathfrak{z}))^2}\right). \label{jk}
	\end{align}
	Now, by \eqref{2F1sym} and \eqref{2F1 special}, \eqref{jk} becomes
	\begin{align*}
	-\frac{(k-1)!(\overline{z}-\mathfrak{z})^k(\overline{z}-\overline{\mathfrak{z}})^k}{2y^{2k}\mathbbm{y}^k(\cosh(d(z,\mathfrak{z}))^2-1)^k}.
	\end{align*}
		Taking $\mathfrak{z}=\tau_Q$, using \eqref{eqn:coshrat} and \eqref{eqn:yQval}, and plugging in termwise, gives \eqref{eqn:RkGreens}.

	The statement for the $\xi$-operator now follows from Corollary \ref{diffop}, using \eqref{XiR} and \eqref{eqn:RkGreens}.
	
	We next compute the image of $\mathcal F_{\mathcal A}$ under $\mathcal{D}^{2k-1}$.  Using Bol's identity \eqref{Bol} we have, by Corollary \ref{diffop},  the equality
	\begin{align*}
	\mathcal{D}^{2k-1}\!\left(\mathcal F_{\mathcal{A}}(z)\right)=-\frac{1}{(4\pi)^{2k-1}}	\frac{a_{k,2k-1}}{2\omega_{\tau_{Q_0}}}R_{0}^k\left(G_k\left(z,\tau_{Q_0}\right)\right).
	\end{align*}
	Using \eqref{eqn:RkGreens} then gives 
the claim.

Finally, Corollary \ref{diffop} immediately implies that $\mathcal{F}_{\mathcal{A}}$ is modular of weight $2-2k$, and the decomposition $\Delta_{2-2k}=-\xi_{2k}\circ \xi_{2-2k}$ together with \eqref{eqn:xiDG} give that $\mathcal{F}_{\mathcal{A}}$ is annihilated by that operator.  From this one concludes that $\mathcal{F}_{\mathcal{A}}$ is a polar harmonic Maass form.
\end{proof}

\subsection{Elliptic expansion of $\mathcal F_{\mathcal{A}}$}

Before stating the elliptic expansion of $\mathcal{F}_{\mathcal{A}}$, we first give a general formula for functions annihilated by the Laplace operator.
\begin{lemma}\label{lem:F+ellexpraise}
Suppose that $\mathcal{M}:\H\to\C$ satisfies $\Delta_{2-2k}(\mathcal M)=0$ and has a singularity of finite order at $\varrho\in\mathbb H$. Then for $0\leq r_{\varrho}(z)\ll_{\varrho} 1$ we have
	\begin{equation}\label{id1}
	\left(\mathcal{M}_{\varrho}^+-\ppart_{\mathcal{M},\varrho}^+\right)(z) = (2i\eta)^{2-2k} \left(z-\overline{\varrho}\right)^{2k-2}\sum_{n\geq 0} \frac{\eta^n}{n!} R_{2-2k}^{n}\left(\mathcal{M}-\ppart_{\mathcal{M},\varrho}\right)(\varrho) X_{\varrho}(z)^n.
	\end{equation}
	In particular, if $\mathcal{M}$ does not have a singularity at $\varrho$, then
	for $0\leq r_{\varrho}(z)\ll_{\varrho} 1$
	we have
	\begin{equation}\label{id2}
	\mathcal{M}_{\varrho}^+(z) = (2i\eta)^{2-2k} \left(z-\overline{\varrho}\right)^{2k-2}\sum_{n\geq 0} \frac{\eta^n}{n!} R_{2-2k}^{n}\left(\mathcal{M}(\varrho)\right) X_{\varrho}(z)^n.
	\end{equation}
	\end{lemma}
\begin{remarks}
\noindent

\noindent
\begin{enumerate}[leftmargin=*]
\item 
In \eqref{id1}, one first needs to act on an independent variable and then plug in $\varrho$, because 
$\mathcal{M}-\ppart_{\mathcal{M},\varrho}$ depends on $\varrho$.  In \eqref{id2}, one may directly apply raising.
\item
	A similar statement is true for the non-meromorphic part $\mathcal{M}_{\varrho}^-$, where raising is instead applied to the conjugate of $\mathcal{M}$.  However, we do not work out the details here.
\end{enumerate}
\end{remarks}
\begin{proof}[Proof of Lemma \ref{lem:F+ellexpraise}]
	Since $\mathcal{M}_{\varrho}^+-\ppart_{\mathcal{M},\varrho}^+$ is holomorphic in some region around $\varrho$, 
Lemma \ref{lem:ellexpraise} provides, for $z$ in some neighborhood of $\varrho$, the expansion
\begin{equation*}
\left(\mathcal{M}_{\varrho}^+-\ppart_{\mathcal{M},\varrho}^+\right)(z)=(2i\eta)^{2-2k} \left(z-\overline{\varrho}\right)^{2k-2}\sum_{n\geq 0} \frac{	\eta^n}{n!} R_{2-2k}^{n}\left(\mathcal{M}_{\varrho}^+-\ppart_{\mathcal{M},\varrho}^+\right)(\varrho) X_{\varrho}(z)^n.
\end{equation*}
The claim hence follows once we prove that
	\[
	R_{2-2k}^n\left(\mathcal{M}_{\varrho}^+-\ppart_{\mathcal{M},\varrho}^+\right)(\varrho)=R_{2-2k}^n\left(\mathcal{M}-\ppart_{\mathcal{M},\varrho}\right)(\varrho).
\]
Noting which terms in the expansion \eqref{eqn:expw} grow as $z\to\varrho$, it suffices to show that for all $m\in \N$ and $n\in\N_0$ 
one has 
	\begin{equation}\label{eqn:raisevanish}
	\left[R_{2-2k,z}^n\!\left((z-\overline{\varrho})^{2k-2}\beta_0\!\left(1-r_{\varrho}(z)^2;2k-1,m\right)X_{\varrho}(z)^{-m}\right)\right]_{z=\varrho}=0.
	\end{equation}
	To prove \eqref{eqn:raisevanish}, we first use \eqref{betid} to rewrite
	\begin{multline}\label{eqn:toraise}
	(z-\overline{\varrho})^{2k-2}\beta_0\!\left(1-r_{\varrho}(z)^2;2k-1,m\right)X_{\varrho}(z)^{-m} \\
	=\!\!\!\! \sum_{0\leq j\leq 2k-2} \binom{2k-2}{j} \frac{(-1)^{j+1}}{j+m} (z-\overline{\varrho})^{2k-2} X_{\varrho}(z)^j \overline{X_{\varrho}(z)^{j+m}}.
	\end{multline}
	We next apply  $R_{2-2k,z}^{n}$ to \eqref{eqn:toraise}. All of the factors other than $(z-\overline{\varrho})^{2k-2} X_{\varrho}(z)^j$ are annihilated by differentiation in $z$.  Note moreover that $\overline{X_{\varrho}(z)^{j+m}}$ vanishes at $z=\varrho$, and 
also that the limit of $R_{2-2k,z}^{n}((z-\overline{\varrho})^{2k-2}X_{\varrho}(z)^j)$ as $z\to \varrho$ exists because $j\geq 0$ and the resulting function is a polynomial (of degree at most $2k-2$) in $z$ with coefficients depending on $y$ and $\varrho$.  Therefore \eqref{eqn:raisevanish} follows.
\end{proof}

We next describe the principal part of the elliptic expansion of $\mathcal F_{\mathcal{A}}$ around $\varrho$, and relate the coefficients of its expansion to higher Green's functions.

\begin{lemma}\label{ellipticF}
The principal part of $\mathcal F_{\mathcal{A}}$ around $\varrho\in\H$ is  
\[
\ppart_{\mathcal F_{\mathcal{A}},\varrho}(z)=\delta_{\varrho=\tau_Q}\mathbb{P}_{1-k,-D,Q}(z),
\]
where here by $\delta_{\varrho=\tau_Q}$ we mean that $\varrho=\tau_Q$ as points  in $\H$ instead of $\SL_2(\Z)\backslash\H$ as used throughout the paper.
The elliptic coefficients of the meromorphic part of $\mathcal F_{\mathcal{A}}$ are given by 
\[
c_{\mathcal F_{\mathcal{A}},\varrho}^{+}(n)=\frac{b_{k,n}}{\omega_{\tau_{Q_0}} }
\begin{cases}
\eta^{2-2k+n}R_{0,\varrho}^{n+1-k}\!\left(G_k^{\operatorname{reg}}\!\left(\varrho,\tau_{Q_0}\right)\right)&\text{if }n\geq k-1,\\[0.1cm]
\eta^{-n} \overline{R_{0,\varrho}^{k-1-n}\!\left(G_k^{\operatorname{reg}}\!\left(\varrho,\tau_{Q_0}\right)\right)} &\text{if }n\leq k-1,
\end{cases}
\]
where $b_{k,n}$ is defined in \eqref{eqn:bkndef}.
\end{lemma}

\begin{proof}
	First note that since $0<D\leq Q_z^2$ with $D=Q_z^2$ if and only 
if $Q(z,1)=0$ by
 \eqref{eqn:Qrewrite}, and since the only possible singularities of $\beta(x;a,b)$ are at $x=0$, $x=1$, and $x\to\infty$, the only terms contributing singularities are those with $\varrho=\tau_Q$.  Since $Q$ is entirely determined by $\tau_Q$ and $D$, it remains to show that $\mathbb{P}_{1-k,-D,Q}$ is precisely a principal part (i.e., that its elliptic expansion \eqref{eqn:expw} only contains terms that grow towards $\varrho$). The claim hence follows once we show that
	\begin{equation}\label{eqn:BetaRels}
	\mathbb{P}_{1-k,-D,Q}(z)=2^{2-3k} \!\left(z-\overline{\tau_{Q}}\right)^{2k-2} v_Q^{1-k} \beta_0\left(1-r_{\tau_{Q}}(z)^2; 2k-1,1-k\right) X_{\tau_{Q}}(z)^{k-1}.
	\end{equation}
To obtain \eqref{eqn:BetaRels}, note that it is not hard to see that the constant $\mathcal C_{2k-1,1-k}$ defined in \eqref{eqn:beta0def} 
vanishes,  and thus
\[
\beta_0(Z;2k-1,1-k)=\beta(Z;2k-1,1-k).
\]
By \eqref{rewriteQ}, the right-hand side of \eqref{eqn:BetaRels} thus equals 
\[
2^{1-2k} D^{\frac{1-k}{2}} Q(z,1)^{k-1} \beta\left(1-r_{\tau_{Q}}(z)^2; 2k-1,1-k\right).
\]
We then use \eqref{betatran}, noting that \eqref{1r} and \eqref{eqn:Qrewrite} imply that 
\begin{equation*}
-\frac{\left(1-r_{\tau_Q}(z)^2\right)^2}{4r_{\tau_Q}(z)^2}= -\frac{Dy^2}{|Q(z,1)|^2},
\end{equation*}
to obtain
\[
2^{-1}i(-1)^k D^{\frac{1-k}{2}}Q(z,1)^{k-1}  \beta\!\left(\frac{-Dy^2}{|Q(z,1)|^2}; k-\frac{1}{2},\frac{1}{2}\right).
	\]
Recalling the definition of $\mathbb{P}_{1-k,-D,Q}$ in  \eqref{defineP}, this yields the statement for the principal part.
%
	
	We next evaluate the elliptic coefficients of the meromorphic part. For $n\in\N_0$, 
Lemma \ref{lem:F+ellexpraise} allows us
 to rewrite
	\begin{equation*}
	c_{\mathcal F_\mathcal{A},\varrho}^+(n)= \frac{ (2i)^{2-2k}}{n!} \eta^{2-2k+n} R_{2-2k}^{n}\left(\mathcal F_{\mathcal{A}}-\ppart_{\mathcal F_{\mathcal{A}},\varrho}\right)(\varrho).
	\end{equation*}	
Using \eqref{eqn:BetaRels} and acting termwise yields
	\begin{equation*}
	c_{\mathcal F_\mathcal{A},\varrho}^+(n)= \frac{ (2i)^{2-2k}}{n!}\eta^{2-2k+n}
  R_{2-2k}^{n}\left(\mathcal F_{\mathcal{A}}^{\operatorname{reg}}(\varrho)\right).
	\end{equation*}
	The result then follows from Corollary \ref{diffop}. 	
\end{proof}

\section{Proof of Theorem \ref{generalint} and Corollary \ref{cor:Greensinner}}\label{sec:residue}

 In this section we prove a more general version of Theorem \ref{generalint} in Theorem \ref{thm:innerGreensGeneral} below, and then use this to prove Corollary \ref{cor:Greensinner}. In order to do so, we first rewrite the inner product $\langle f,f_{\mathcal A}\rangle$ in terms of the elliptic coefficients of $f$ given in \eqref{eqn:fEllExp}, as well as those of $\mathcal F_{\mathcal{A}}$, evaluated explicitly in Lemma \ref{ellipticF}.

\begin{theorem}\label{thm:wnotz}
If $f\in \mathbb{S}_{2k}$ has its poles at $\mathfrak{z}_{1},\dots, \mathfrak{z}_r$ in $\SL_2(\Z)\backslash\H$, then
$$
\left<f,f_{\mathcal{A}}\right>=\pi \sum_{\ell=1}^{r} \frac{1}{\mathbbm{y}_\ell \omega_{\mathfrak{z}_\ell}}
 \sum_{n\geq 1} c_{f,\mathfrak{z}_\ell}(-n)c_{\mathcal{F}_{\mathcal{A}},\mathfrak{z}_\ell}^+(n-1).
$$
\end{theorem}
\begin{proof} Since the functions $\left\{ \Psi_{2k,m}(\cdot,\mathfrak{z}): \mathfrak{z}\in\H,\ m\in\Z\right\}$ span $\SS_{2k}$, linearity allows us to assume that $f(z)=(2\omega_{\mathfrak{z}})^{-1}\Psi_{2k,m}(z,\mathfrak{z})$ for $m\in\Z$, $\mathfrak{z}\in\H$.  We use a trick employed by many authors (cf. \cite{Bo1,BruinierFunke,DITRQ}) for rewriting the inner product.  By \eqref{eqn:xiDG} we obtain 
\begin{equation}\label{regist}
\left<f,f_{\mathcal{A}}\right>=\left<f, \xi_{2-2k}\left(\mathcal F_{\mathcal{A}}\right)\right>.
\end{equation}
We take the implied integral in \eqref{regist} over the cut-off fundamental domain  $ F_T^*$, consisting of those $z\in F^*$ for which $z$ is equivalent to a point in $ F_T$ under the action of $\SL_2(\Z)$, and then let $T\to \infty$.  We require a few additional properties of $ F^*$.  
First we may assume, without loss of generality, that $\tau_{Q_0},\mathfrak{z}\in  F^*$. We also claim that since there are no poles of $f$ or $f_{\mathcal{A}}$ for $y\gg 0$, $ F^*$ may be constructed so that for $T\gg 0$, the boundary of $ F_T^*$ includes the line from $-\frac12+iT$ to $\frac12+iT$. Indeed, one can explicitly build $ F^*$ from the standard fundamental domain $ F$ by successively removing partial balls $\mathcal{B}_{\delta}(\mathfrak{z}_\ell)\cap  F$ around each pole $\mathfrak{z}_\ell\in \partial  F$ that is not an elliptic fixed point and moving them to the other side of the fundamental domain with respect to the imaginary axis to combine with other partial balls around equivalent points $\gamma\mathfrak{z}_\ell\in \partial  F$ to form entire balls $\mathcal{B}_{\delta}(\gamma\mathfrak{z}_\ell)$ for some $\gamma\in\SL_2(\Z)$.  Since the part of the fundamental domain with $y\gg 0$ remains unchanged, the boundary of $ F_T^*$ is as desired.  Moreover, we may choose $\delta>0$ sufficiently small such that $\mathcal{B}_{\delta}(\mathfrak{z}_{\ell})$ is contained inside $\Gamma_{\mathfrak{z}_{\ell}}  F^*$ and balls around different points are disjoint.  By Stokes' Theorem, using the meromorphicity of $f$ and the vanishing of $f(z)h_{s,\varrho}(z)$ at $z=\varrho$ for $s\in\C$ with $\sigma\gg 0$,  \eqref{regist} equals 
\begin{multline}\label{eqn:polez}
-\operatorname{CT}_{s=0}\Bigg(\int_{ F^*} f(z)\overline{
\xi_{0}
\left(h_{s_1,\mathfrak{z}}(z)h_{s_2,\tau_{Q_0}}(z)\right)} \mathcal F_{\mathcal A}(z)
dxdy
\\
+\lim_{T\to\infty}\int_{\partial  F_T^*} f(z) h_{s_1,\mathfrak{z}}(z)h_{s_2,\tau_{Q_0}}(z)\mathcal F_{\mathcal A}(z)dz\Bigg).
\end{multline}
Note that the minus sign occurring in the second term in \eqref{eqn:polez} comes from the computation of the exterior derivative in terms of the $\xi$-operator; see the last two formulas on page 12 of \cite{BruinierFunke} for further details.

Recalling the definition after \eqref{eqn:OurReg}, we note that $z\mapsto h_{s_0,\varrho}(z)$ is invariant under $\Gamma$, and hence the integrand in the second term of \eqref{eqn:polez} is modular of weight $2$. Combining this modularity with the exponential decay of $f$ towards $i\infty$ and the polynomial growth of the other factors, one concludes that the second term vanishes as $T\to\infty$.  Using the invariance of the integrand under the action of $\Gamma_{\tau_{Q_0}}$ and $\Gamma_{\mathfrak{z}}$, we then rewrite \eqref{eqn:polez} as
\begin{equation}\label{eqn:polez5}
-\frac{1}{\omega_{\tau_{Q_0}}\omega_{\mathfrak{z}}}\operatorname{CT}_{s=0}\left(\sum_{\gamma_1\in\Gamma_{\mathfrak{z}}}\sum_{\gamma_2\in\Gamma_{\tau_{Q_0}}}\int_{\gamma_1\gamma_2 F^*} f(z)\overline{\xi_{0}\left(h_{s_1,\mathfrak{z}}(z)h_{s_2,\tau_{Q_0}}(z)\right)} \mathcal F_{\mathcal A}(z)dxdy\right).
\end{equation}
Note that for $\varrho\in  F^*$, no other  element of $\Gamma_{\varrho} F^*$ is equivalent to $\varrho$ 
modulo $\SL_2(\Z)$, and 
we have the equality $r_{\varrho}(Mz)=r_{\varrho}(z)$ for every $M\in \Gamma_{\varrho}$ by \eqref{eqn:requal}.  Hence the equality $h_{s_j,\varrho}(z)=r_{\varrho}(z)^{2s_j}$ holds for every $z\in \Gamma_{\varrho} F^*$.
For $(\varrho,s_0)\in\{(\mathfrak{z},s_1),(\tau_{Q_0},s_2)\}$ we may therefore compute 
\begin{equation}\label{eqn:xih}
\overline{\xi_{0}\left(h_{s_0,\varrho}(z)\right)}=-4s_0\eta r_{\varrho}(z)^{2s_0-2} \frac{X_{\varrho}(z)}{\left(\overline{z}-\varrho\right)^2}.
\end{equation}

Thus $1/(s_1s_2)$ times the first integral in \eqref{eqn:polez} restricted to $z\notin \mathcal{B}_{\delta}(\mathfrak{z})\cup \mathcal{B}_{\delta}(\tau_{Q_0})$ converges absolutely and locally uniformly in $s$, and hence the corresponding contribution to the integral is analytic and vanishes at $s=0$.  To evaluate the remaining part of  \eqref{eqn:polez5}, we first compute $\xi_{0}(h_{s_1,\mathfrak{z}}(z)h_{s_2,\tau_{Q_0}}(z))$ for $z\in \mathcal{B}_{\delta}(\mathfrak{z})\cap  F^*$ using the product rule.  The term coming from differentiating $h_{s_2,\tau_{Q_0}}$ vanishes in the limit $s_2\to 0$ by \eqref{eqn:xih}.  We then use \eqref{eqn:Ballsplit} and \eqref{eqn:polez5} to show that the first integral in \eqref{eqn:polez} over $\mathcal{B}_{\delta}(\mathfrak{z})\cap  F^*$ equals
$$
-\frac{1}{\omega_{\mathfrak{z}}}\operatorname{CT}_{s_1=0}\left(\int_{\mathcal{B}_\delta(\mathfrak{z})} f(z)\overline{\xi_{0,z}\left(h_{s_1,\mathfrak{z}}(z)\right)} \mathcal F_{\mathcal A}(z) dxdy\right).
$$

Plugging \eqref{eqn:xih} into the latter expression, and repeating the argument for $\tau_{Q_0}$ if $\tau_{Q_0}\neq \mathfrak{z}$, \eqref{eqn:polez} becomes
\begin{equation}
\label{eqn:polez2}
\frac{4v_{Q_0}}{\omega_{\tau_{Q_0}}}\mathcal{J}\!\left(\tau_{Q_0}\right)+\delta_{\mathfrak{z}\neq \tau_{Q_0}}\frac{4\mathbbm{y}}{\omega_{\mathfrak{z}}}\mathcal{J}(\mathfrak{z}), 
\end{equation}
where
\begin{equation}
\label{eqn:inttoeval}
\mathcal{J}(\varrho):=
\operatorname{CT}_{s_0=0}\left(s_0 \int_{\mathcal{B}_{\delta}(\varrho)} f(z)r_{\varrho}(z)^{2s_0-2}\frac{X_{\varrho}(z)}{\left(\overline{z}-\varrho\right)^2} \mathcal F_{\mathcal A}(z) dxdy
\right).
\end{equation} 

To evaluate $\mathcal{J}(\varrho)$, we insert the elliptic expansion \eqref{eqn:Psiexp} of $f(z)=(2\omega_{\mathfrak{z}})^{-1}\Psi_{2k,m}(z,\mathfrak{z})$ around $\varrho$ and the expansion of $\mathcal F_{\mathcal A}$ using the explicit principal part given in Lemma \ref{ellipticF} (rewritten as 
in \eqref{eqn:BetaRels}) to
 see that the integral in \eqref{eqn:inttoeval} equals
\begin{multline}\label{eqn:intb4change}
\frac{1}{\eta^2}\int_{\mathcal{B}_{\delta}(\varrho)}\frac{\eta^2}{\left|z-\overline{\varrho}\right|^4}\left(\delta_{\varrho=\mathfrak{z}} X_{\varrho}(z)^m +\sum_{n\geq 0}  c(n) X_{\varrho}(z)^n\right)r_{\varrho}(z)^{2s_0-2}X_{\varrho}(z)\\
\times \Bigg( 2^{2-3k}v_{Q_0}^{1-k}\delta_{\varrho=\tau_{Q_0}}\beta_0\left(1-r_{\tau_{Q_0}}(z)^2;2k-1,1-k\right) X_{\tau_{Q_0}}(z)^{k-1} +\sum_{\ell\geq 0}  c_{\mathcal{F}_{\mathcal{A}},\varrho}^+(\ell) X_{\varrho}(z)^{\ell}\\
 + \sum_{\ell<0} c_{\mathcal{F}_{\mathcal{A}},\varrho}^-(\ell) \beta_0\left(1-r_{\varrho}(z)^2;2k-1,-\ell\right)X_{\varrho}(z)^{\ell} \Bigg)dxdy.
\end{multline}
Making the change of variables $X_{\varrho}(z)=Re^{i\theta}$ and noting that $\frac{\eta^2}{\left|z-\overline{\varrho}\right|^{4}} dxdy = \frac{R}{4} d\theta dR$
and $r_{\varrho}(z)=R$,
we may rewrite \eqref{eqn:intb4change} as
\begin{multline*}
\frac{1}{4\eta^2}\int_{0}^{\delta}\int_{0}^{2\pi}\left( \delta_{\varrho=\mathfrak{z}}R^m e^{im\theta} +\sum_{n\geq 0} c(n) R^n e^{in\theta}\right)
\Bigg(\delta_{\varrho=\tau_{Q_0}}\frac{ R^{k-1+2s_0}e^{ik\theta}}{2^{3k-2}v_{Q_0}^{k-1}}\beta_0\left(1-R^2;2k-1,1-k\right) \\
+\sum_{\ell\geq 0} c_{\mathcal{F}_{\mathcal{A}},\varrho}^+(\ell) R^{\ell+2s_0}e^{i(\ell+1)\theta}+ \sum_{\ell<0} c_{\mathcal{F}_{\mathcal{A}},\varrho}^-(\ell)\beta_0\left(1-R^2;2k-1,-\ell\right)R^{\ell+2s_0}e^{i(\ell+1) \theta}  \Bigg) d\theta dR.
 \end{multline*}
Expanding, the integral over $\theta$ vanishes unless the power of $e^{i\theta}$ is zero.
The latter expression thus equals
\begin{multline}\label{eqn:polez3}
\frac{\pi }{2\eta^2} \int_{0}^{\delta} \bigg(\frac{\delta_{m=-k}\delta_{\varrho=\tau_{Q}=\mathfrak{z}}}{2^{3k-2}v_{Q_0}^{k-1}}\beta_0\left(1-R^2;2k-1,1-k\right) +  \delta_{\varrho=\mathfrak{z}} c_{\mathcal{F}_{\mathcal{A}},\varrho}^+
\left(-m-1\right)\\
+\sum_{n\geq 0} \left(c(n)+\delta_{n=m}\delta_{\varrho=\mathfrak{z}}\right)
c_{\mathcal{F}_{\mathcal{A}},\varrho}^-
\left(-n-1\right) \beta_0\left(1-R^2;2k-1,n+1\right) \bigg)R^{2s_0-1} dR.
 \end{multline}
\noindent
To determine the residue of \eqref{eqn:polez3} at $s_0=0$, we use \eqref{betid} with $a=2k-1$ and $b=-\ell$ to expand $\beta_0(1-R^2;2k-1,-\ell)$. For $\sigma_0\gg 0$, multiplying the first term in \eqref{betid} by $R^{2s_0-1}$ and integrating then yields
$$
\sum_{\substack{0\leq j\leq 2k-2\\ j\neq \ell}} \binom{2k-2}{j}\frac{(-1)^{j+1}}{j-\ell}\int_{0}^{\delta}R^{2(j-\ell+s_0)-1} dR = \sum_{\substack{0\leq j\leq 2k-2\\ j\neq \ell}} \binom{2k-2}{j}\frac{(-1)^{j+1} \delta^{2(j-\ell+s_0)}}{2(j-\ell)(j-\ell+s_0)},
$$
which is holomorphic at $s_0=0$, so the corresponding terms in \eqref{eqn:polez3} give no residue.  Hence

 \begin{equation}\label{eqn:polez4}
 \mathcal{J}(\varrho)=\frac{\pi  \delta_{\varrho=\mathfrak{z}}}{2\eta^2
}\operatorname{CT}_{s_0=0}\left( -s_0\frac{\delta_{m=-k}\delta_{\varrho=\tau_{Q_0}}}{(-8v_{Q_0})^{k-1}}\binom{2k-2}{k-1}\int_{0}^{\delta} \log(R)R^{2s_0-1}dR +
c_{\mathcal{F}_{\mathcal{A}},\varrho}^+
\left(-m-1\right)\frac{\delta^{2s_0}}{2}\right).
 \end{equation}
\noindent
Using integration by parts for the first summand 
in \eqref{eqn:polez4}, we
 obtain a meromorphic continuation with no constant term, as
$$
s_0\int_{0}^{\delta} \log\left(R\right)R^{2s_0-1}dR = \frac{\delta^{2s_0}}{2}\log(\delta) -\frac{1}{4s_0} \delta^{2s_0}
=-\frac{1}{4s_0} + O(s_0).
$$
Therefore 
\[
\mathcal{J}(\varrho)=\frac{\pi \delta_{\varrho=\mathfrak{z}}}{4\eta^2}c_{\mathcal{F}_{\mathcal{A}},\varrho}(-m-1).
\]
Plugging 
this 
back into \eqref{eqn:polez2} and recalling that this equals \eqref{regist} then gives 
\begin{align*}
\langle f, f_{\mathcal{A}}\rangle & =\left\langle f, \xi_{2-2k}\left(\mathcal{F}_{\mathcal{A}}\right)\right\rangle
=\frac{4v_{Q_0}}{\omega_{\tau_{Q_0}}}\mathcal{J}\left(\tau_{Q_0}\right)+\frac{4\mathbbm{y}\delta_{\mathfrak{z}\neq \tau_{Q_0}}}{\omega_{\mathfrak{z}}}\mathcal{J}(\mathfrak{z})\\
&\ =\frac{\pi}{v_{Q_0}\omega_{\tau_{Q_0}}}\delta_{\mathfrak{z}=\tau_{Q_0}}c^+_{\mathcal{F}_{\mathcal{A}, \tau_{Q_0}}}(-m-1)
+\frac{\pi}{\mathbbm{y}\omega_{\mathfrak{z}}}\delta_{\mathfrak{z}\neq\tau_{Q_0}}c^+_{\mathcal{F}_{\mathcal{A}, \mathfrak{z}}}(-m-1)=\frac{\pi}{\mathbbm{y}\omega_{\mathfrak{z}}}
c^+_{\mathcal{F}_{\mathcal{A}, \mathfrak{z}}}(-m-1).
\end{align*}
\end{proof}

The following Theorem generalizes Theorem \ref{generalint} to also allow poles at $\tau_{Q_0}$.
\begin{theorem}\label{thm:innerGreensGeneral}
If $Q_0\in \mathcal{A}\in\mathcal{Q}_{-D}\backslash \SL_2(\Z)$ and $f\in \mathbb{S}_{2k}$ with poles in $\SL_2(\Z)\backslash \H$ at $\mathfrak{z}_1,\dots,\mathfrak{z}_r$, then
\begin{multline*}
\left<f,f_{\mathcal{A}}\right>=\frac{\pi}{\omega_{\tau_{Q_0}}}
 \sum_{\ell=1}^r \frac1{ \omega_{\mathfrak{z}_\ell}}
\Bigg(\sum_{n\geq k}b_{k,n-1} \mathbbm{y}_\ell^{-2k+n} c_{f,\mathfrak{z}_{\ell}}(-n)  R_{0}^{n-k}\left(G_k^{\operatorname{reg}}(z,\tau_{Q_0})\right)\\
+ \sum_{n=1}^{k-1} b_{k,n-1}  \mathbbm{y}_\ell^{-n} c_{f,\mathfrak{z}_{\ell}}(-n) \overline{R_0^{k-n}\left(G_k^{\operatorname{reg}}(z,\tau_{Q_0})\right)} \Bigg).
\end{multline*}
\end{theorem}
\begin{proof}
The result follows directly by plugging Lemma \ref{ellipticF} into the statement of Theorem \ref{thm:wnotz}.
\end{proof}

We finally prove Corollary \ref{cor:Greensinner}.
\begin{proof}[Proof of Corollary \ref{cor:Greensinner}]
This follows immediately from Theorem \ref{generalint} and Lemma \ref{lem:fPsi}.
\end{proof}
\section{Future questions}\label{sec:future}

\noindent
We conclude the paper by discussing some possible future directions that one could pursue:
\noindent
\begin{enumerate}[leftmargin=*]
\item
Note that by Theorem \ref{generalint}, $f_{\mathcal{A}}$ is orthogonal to cusp forms, which was also proven by Petersson \cite{Pe2}.
  Combining the regularizations for growth towards the cusps and towards points in $\H$, one can further prove that $f_{\mathcal{A}}$ is orthogonal to weakly holomorphic modular forms, but we do not carry out the details here. After reading a preliminary version of this paper, Zemel \cite{Zemel} considered some questions related to inner products between weakly holomorphic modular forms and meromorphic cusp forms.  
\item
Images of lifts between integral and half-integral weight weak Maass forms  have Fourier expansions that  can be written as CM-traces for negative discriminants and cycle integrals for positive discriminants  \cite{BO,BruO,DIT}. Thus the appearance of CM-values of $\SL_2(\Z)$-invariant functions in Theorem \ref{thm:innerGreensGeneral} is natural. Since the generating function of Zagier's cusp forms for positive discriminants yields the (holomorphic) kernel function for the first Shintani lift, one may ask whether there is a connection between CM-traces and the generating function of the $f_{k,-D}$.  However, the naive generating function diverges, and furthermore, it would have a dense set of poles in the upper half-plane.  It hence might be interesting to investigate instead whether the generating function for the regularized function $f_{k,-D}^{\operatorname{reg}}$ has any connection to CM-traces.
\item In light of the connection in Corollary \ref{cor:Greensinner}, it would be interesting to investigate Conjecture 4.4 of \cite{GZ}, concerning $G_k$ evaluated at CM-points.  Moreover, since $G_k^{\operatorname{reg}}(\tau_{Q_0},\tau_{Q_0})$ naturally appears when computing $\left<f_{\mathcal{A}},f_{\mathcal{A}}\right>$, one can probably use the regularized higher Green's functions to reformulate the conjecture to include the case when the CM-points agree.  Given the connections to heights and geometry in \cite{GZ} and \cite{Zhang}, it would also be interesting to see if the identity in Corollary \ref{cor:Greensinner} holds for $k=1$ and higher level in this case.
\item In Conjecture 4.4 of \cite{GZ}, Gross and Zagier took linear combinations of $G_{k}$ acted on by 
Hecke operators, and
 conjectured that these linear combinations evaluated at CM-points are essentially logarithms of algebraic numbers whenever the linear combinations satisfy certain relations.  These relations are determined by linear equations defined by the Fourier coefficients of weight $2k$ cusp forms.  
Note that by Corollary \ref{diffop}, $G_{k}(z,\tau_Q)$ is essentially $R_0^{k-1}(\mathcal{F}_{\mathcal{A}}(z))$, while $\mathcal F_{\mathcal{A}}$ is naturally related to $f_{\mathcal{A}}$ via differential operators in Theorem \ref{thm:Gpolar} (2).  Translating the condition of Gross and Zagier into a condition on polar harmonic Maass forms might be enlightening in two directions.  On the one hand, it might carve out a natural subspace of weight $2k$ meromorphic modular forms (corresponding to 
the image under $\xi_{2-2k}$
 of those 
polar harmonic Maass forms 
 satisfying these conditions), which
 may satisfy other interesting properties.  On the other 
hand,
 by applying the theory of harmonic Maass forms, one may be able to loosen the conditions and investigate what happens for general linear combinations.
\end{enumerate}

\end{document}